%% file: HEInfRigSm.tex
\documentclass[11pt]{amsart}
\usepackage{amsaddr}
\usepackage{amsmath,amssymb,amscd,amsthm}

\usepackage{graphics}
\usepackage{color}

\usepackage{url}

\newtheorem{dfn}{Definition}[section]
\newtheorem{thm}[dfn]{Theorem}

\newtheorem{lem}[dfn]{Lemma}
\newtheorem{cor}[dfn]{Corollary}

\newtheorem*{weylpr}{Weyl problem}
\newtheorem*{minkpr}{Minkowski problem}

\newtheorem{conj}[dfn]{Conjecture}

\theoremstyle{definition}
\newtheorem{rem}[dfn]{Remark}
\newtheorem{exl}[dfn]{Example}
\newtheorem{prb}{Problem}

\theoremstyle{plain}

\numberwithin{equation}{section}

\def\R{{\mathbb R}}

\def\Sph{{\mathbb S}}

\def\Ball{{\mathbb B}}

\def\H{{\mathbb H}}
\def\dS{\mathrm{d}{\mathbb S}}

\def\phi{\varphi}
\def\epsilon{\varepsilon}

\newcommand{\Area}{\operatorname{Area}}
\newcommand{\cone}{\operatorname{cone}}

\newcommand{\Vol}{\operatorname{Vol}}

\def\rk{\operatorname{rk}}
\def\can{\mathit{can}}
\def\const{\mathrm{const}}
\def\HE{\operatorname{HE}}
\def\scal{\operatorname{scal}}
\def\wnabla{\widetilde{\nabla}}
\def\wR{\widetilde{R}}
\def\wg{\widetilde{g}}
\def\id{\operatorname{id}}
\def\tr{\operatorname{tr}}
\def\Hess{\operatorname{Hess}}

\def\darea{\operatorname{darea}}
\def\dvol{\operatorname{dvol}}
\def\End{\operatorname{End}}
\def\Hom{\operatorname{Hom}}
\def\sgn{\operatorname{sgn}}
\def\Alt{\operatorname{Alt}}
\def\Sym{\operatorname{Sym}}
\def\I{\operatorname{I}}
\def\II{\operatorname{II}}
\def\III{\operatorname{III}}
\def\cG{\mathcal{G}}

\renewcommand{\det}{\operatorname{det}}
\newcommand{\bigdot}{{\displaystyle{\cdot}}}

\title[Rigidity of surfaces and the Hilbert-Einstein functional II]{Infinitesimal rigidity of convex surfaces through the second derivative of the Hilbert-Einstein functional\\ II: Smooth case}
\author{Ivan Izmestiev}
\address{Institut f\"ur Mathematik, MA 8-3 \\
Technische Universit\"at Berlin \\
Stra{\ss}e des 17. Juni 136 \\
D-10623 Berlin, Germany}
\email{izmestiev@math.tu-berlin.de}
\thanks{Supported by the DFG Research Unit 565 ``Polyhedral Surfaces''}
\date{May 25, 2011}

\begin{document}

\begin{abstract}
The paper is centered around a new proof of the infinitesimal rigidity of smooth closed surfaces with everywhere positive Gauss curvature. We use a reformulation that replaces deformation of an embedding by deformation of the metric inside the body bounded by the surface. The proof is obtained by studying derivatives of the Hilbert-Einstein functional with boundary term.

This approach is in a sense dual to proving the Gauss infinitesimal rigidity, that is rigidity with respect to the Gauss curvature parametrized by the Gauss map, by studying derivatives of the volume bounded by the surface. We recall that Blaschke's classical proof of the infinitesimal rigidity is also related to the Gauss infinitesimal rigidity, but in a different way: while Blaschke uses Gauss rigidity of the same surface, we use the Gauss rigidity of the polar dual.

In the spherical and in the hyperbolic-de Sitter space, there is a perfect duality between the Hilbert-Einstein functional and the volume, as well as between both kinds of rigidity.

We also indicate directions for future research, including the infinitesimal rigidity of convex cores of hyperbolic 3--manifolds.
\end{abstract}

\maketitle

\setcounter{tocdepth}{1}
\tableofcontents

\section{Introduction}
\subsection{Infinitesimal rigidity of strictly convex smooth surfaces}
\label{subsec:IntroThm}
Let $M \subset \R^3$ be a smooth closed surface. An \emph{infinitesimal deformation} of $M$ is a vector field along $M$, which can be viewed as a map $\xi \colon M \to \R^3$. An infinitesimal deformation produces a family of embeddings (for small $t$),
$$
\phi_t \colon M \to \R^3, \quad x \mapsto x + t\xi.
$$
Denote by $g_t$ the metric induced on $M$ by $\phi_t$. A deformation $\xi$ is called \emph{isometric}, if $\left. \frac{d}{dt} \right|_{t=0} g_t = 0$. Every surface has \emph{trivial} isometric infinitesimal deformations, which are restrictions to $M$ of Killing vector fields in $\R^3$.

We call a surface with everywhere positive Gauss curvature \emph{strictly convex}.
\begin{quote}
Every strictly convex smooth closed surface $M$ is infinitesimally rigid, that is every isometric infinitesimal deformation of $M$ is trivial.
\end{quote}
This theorem was proved by Liebmann \cite{Lieb99} for analytic deformations of analytic surfaces and by Blaschke \cite{Bla12, Bla21} and Weyl \cite{Weyl17} for $C^3$--deformations of $C^3$--surfaces. For references to other proofs, see \cite[\S 4.1]{IKS94}.

In this paper we present still another proof. We also describe the general framework, historical perspective, and possible further developments.

\subsection{The approach}
\label{subsec:IntroApproach}
The infinitesimal rigidity theorem stated above deals with a deformation $\phi_t$ of an embedding $M \subset \R^3$. However, it is possible to reformulate it in terms of deformations of the metric inside the body $P \subset \R^3$ bounded by $M$.

View $P$ as an abstract Riemannian manifold equipped with a flat metric $G$ induced from $\R^3$. An \emph{infinitesimal deformation} of $G$ is a field $\dot G$ of symmetric bilinear forms on $P$. An infinitesimal deformation is called \emph{curvature preserving}, if the Riemannian metric $G_t = G + t \dot G$ is flat in the first order of~$t$. If $\dot G$ is curvature preserving and vanishes on vectors tangent to the boundary, then it induces an isometric infinitesimal deformation $\xi$ of the embedding $M \subset \R^3$, cf. \cite[Proposition 3]{Cal61} for a local statement. An infinitesimal deformation $\dot G$ is \emph{trivial}, if it pulls back the metric $G$ by an infinitesimal diffeomorphism that restricts to the identity on the boundary: $\dot G = {\mathcal L}_\eta G$ for some vector field $\eta$ on $P$ such that $\eta|_{\partial P} = 0$.
\begin{quote}
A surface $M \subset \R^3$ is infinitesimally rigid $\Leftrightarrow$ every curvature preserving deformation $\dot G$ such that $\dot G|_{TM} = 0$ is trivial.
\end{quote}

The space $\cG_{\mathrm{triv}}$ of trivial infinitesimal deformations is very large. Therefore it is convenient to consider only $\dot G$ from a subspace $\cG_0 \subset \cG$ in the space of all infinitesimal deformations such that $\cG_{\mathrm{triv}} + \cG_0 = \cG$ and such that the intersection $\cG_{\mathrm{triv}} \cap \cG_0$ is reasonably small. We take as $\cG_0$ the tangent space to the space of warped product metrics of the following form:
$$
\widetilde{g}_r = d\rho^2 + \rho^2 \, \left( \frac{g - dr \otimes dr}{r^2} \right).
$$
Here $g$ is the Riemannian metric on $M \subset \R^3$, and $r \colon M \to \R_+$ is a smooth function, so that $\widetilde{g}_r$ is a Riemannian metric on $\R_+ \times M \cong \R^3 \setminus \{0\}$. The metric $\wg_r$ is flat, if $r$ measures the distance from $0 \in \R^3$ (assumed to lie within $P$). For every $r$, the restriction of $\widetilde{g}_r$ to the surface $\rho = r(x)$ is $g$.

The problem of isometric infinitesimal deformations can now be reformulated as describing those variations $\dot r$ of the function $r$ that leave the metric $\widetilde{g}_r$ flat in the first order. Trivial variations $\dot r$ arise from moving the coordinate origin and hence form a space of dimension~$3$. Note that the origin may be a singular point of the metric $\widetilde{g}_r$, so that our space $\cG_0$ is not really a subspace of $\cG$.

The curvature of a warped product metric $\widetilde{g}_r$ can be described by a single function $\sec \colon \Sph^2 \to \R$ (the sectional curvature in the planes tangent to $M$). Therefore the infinitesimal rigidity of $M$ is equivalent to
\begin{equation*}
\label{eqn:IntroSecDot}
\mbox{If }\sec^\bigdot = 0, \mbox{ then }\dot r \mbox{ is trivial.}
\end{equation*}

\subsection{Infinitesimal rigidity of Einstein manifolds}
\label{subsec:IntroKoiso}
In \cite{Koi78}, Koiso proved the infinitesimal rigidity of compact closed Einstein manifolds under certain restrictions on curvature. His method can be described as the Bochner technique applied to the second derivative of the Hilbert-Einstein functional. Our proof of the infinitesimal rigidity of strictly convex closed surfaces emulates Koiso's argument in a certain sense. Koiso also chooses a subspace $\cG_0$ in the space of all infinitesimal deformations of the metric, but our choice of $\cG_0$ is very specific and works only for manifolds homeomorphic to the ball.

\subsection{New proof of the infinitesimal rigidity of strictly convex surfaces}
\label{subsec:IntroOurProof}
The \emph{Hilbert-Ein\-stein functional} on the space of Riemannian metrics on a $3$--dimensional compact manifold $P$ with boundary is defined as
$$
\HE(\wg) = \frac12 \int\limits_P \scal \, \dvol + 2 \int\limits_{\partial P} H \, \darea,
$$
where $\scal$ is the scalar curvature of the metric $\wg$, and $H$ is the mean curvature (half of the trace of the shape operator), cf. \cite{Esc96,Ara03}. In our situation $P$ is a ball, and the metric $\wg = \wg_r$ is a warped product metric, see Subsection \ref{subsec:IntroApproach}. This allows to express $\HE$ as an integral over the boundary $M = \partial P$, see Theorem \ref{thm:HE}.

We prove the following formula for the derivative of $\HE$ in the direction~$\dot r$:
$$
\HE^\bigdot = \int\limits_M \dot r \frac{\sec}{\cos\alpha} \, \darea,
$$
where $\alpha$ is a certain function on $M$ determined by $r$ and $g$ (compare this with the well-known formula for $\HE^\bigdot$ for an arbitrary variation of a metric $\wg$ on a compact closed manifold).

It follows that the second derivative of $\HE$ in the direction $\dot r$ equals
$$
\HE^{\bigdot\bigdot} = \int\limits_M \dot r \left( \frac{\sec}{\cos \alpha} \right)^\bigdot \, \darea.
$$
If the variation $\dot r$ is curvature-preserving, then we have $\sec^\bigdot = 0$. As the initial metric $\wg_r$ is flat, we also have $\sec = 0$. Hence $\HE^{\bigdot\bigdot}$ vanishes. On the other hand, an integration by parts yields
$$
\HE^{\bigdot\bigdot} = \int\limits_M 2h \det \dot B \, \darea,
$$
where $h$ is the support function of $M$ (i.~e. distance from $0$ to tangent planes) and $B$ is the shape operator. A simple algebraic lemma shows that $\det \dot B \le 0$ and that $\det \dot B$ vanishes only if $\dot B$ vanishes. As $h > 0$, it follows that the integral in the last equation vanishes only if $\dot B = 0$. Finally, $\dot B = 0$ implies that $\dot r$ is a trivial variation.

\subsection{Discrete and smooth, classical and modern}
The proof sketched in the previuos subsection looks as a logical development of ideas of Koiso and others. However, we came to it on a different way, namely from its discrete analog \cite{Izm11a}. On the other hand, Blaschke's elegant and somewhat mysterious proof \cite{Bla21} originated from another proof of another theorem that used the second derivative of a certain functional. Therefore our proof can also be seen as a logical development of ideas of Blaschke and others.

It turns out that there are two infinitesimal rigidity theorems, metric (the one stated above) and Gauss. Metric infinitesimal rigidity is related to the Weyl problem in the same way as Gauss infinitesimal rigidity is related to the Minkowski problem. Besides, these rigidity theorems are \emph{in two ways} dual to each other. We now proceed to explaining these statements.

\subsection{Blaschke's proof of the infinitesimal rigidity of strictly convex surfaces}
\label{subsec:IntroBlaschke}
Here is a sketch of the argument from \cite{Bla21} reproduced in Subsection \ref{subsec:BlaschkeProof} in a greater detail.

Every isometric infinitesimal deformation $\xi$ of a surface $M \subset \R^3$ has an associated \emph{rotation vector field} $\eta \colon M \to \R^3$. A deformation is trivial if and only if its rotation field is constant. One can show that the differential $d\eta$ maps $TM$ to itself (similar to the differential of the field of unit normals). Performing an integration by parts, Blaschke proves the equation
\begin{equation}
\label{eqn:IntroHDet}
\int\limits_M 2 h \det d\eta \, \darea = 0,
\end{equation}
where $h$ is the support function of $M$. An ingenious argument shows that $\det d\eta$ is non-positive, and vanishes only if $d\eta$ does. As without loss of generality $h$ is positive, equation \eqref{eqn:IntroHDet} implies $d\eta = 0$. Hence $\eta$ is constant, and the infinitesimal deformation $\xi$ is trivial.

Blaschke's proof can be demystified a bit by observing that $d\eta = J \dot B$, where $J \colon TM \to TM$ is the rotation by $90^\circ$. This explains the above properties of $\det d\eta$ in a more conceptual way. Besides, this relates Blaschke's proof to the proof sketched in Subsection \ref{subsec:IntroOurProof}. However, this relation is not quite straightforward, as we will see in Subsection \ref{subsec:IntroDarboux}.

\subsection{Gauss infinitesimal rigidity and origin of Blaschke's proof}
\label{subsec:IntroBlaschkeOrigin}
Let $M \subset \R^3$ be a strictly convex smooth closed surface, and let $\eta \colon M \to \R^3$ be an infinitesimal deformation of $M$. Assume that the embeddings $\psi_t \colon M \to \R^3$, $\psi_t(x) = x + t \eta(x)$ preserve the Gauss map: $d\psi_t(T_xM) \parallel T_xM$. Clearly, this is equivalent to $d\eta(TM) \subset TM$. Denote by $K_t \colon M \to \R$ the Gauss curvature of the embedding $\psi_t$, and by $\dot K$ the derivative of $K_t$ at $t=0$. We call an infinitesimal deformation $\eta$ \emph{isogauss}, if $d\eta(TM) \subset TM$ and $\dot K = 0$. If $\eta$ is constant, that is $\psi_t$ parallelly translates $M$, then $\eta$ is called a \emph{trivial} isogauss infinitesimal deformation.
\begin{quote}
Every strictly convex smooth closed surface $M$ is Gauss infinitesimally rigid, that is every isogauss infinitesimal deformation of $M$ is trivial.
\end{quote}
This theorem is implicitly contained in Hilbert's treatment \cite{Hil10} of Minkowski's \emph{mixed volumes theory}.

Blaschke \cite{Bla12} observed that certain equations in Hilbert's work coincide with equations of the theory of isometric infinitesimal deformations. Weyl \cite{Weyl17} concretized this observation as follows:
\begin{quote}
The rotation vector field of an isometric infinitesimal deformation of $M$ is an isogauss infinitesimal deformation of $M$.
\end{quote}
Thus the infinitesimal rigidity theorem from Subsection \ref{subsec:IntroThm} can be proved by applying the Minkowski-Hilbert argument to the rotation vector field of the deformation. Such a proof was written down by Blaschke in \cite{Bla21} in a concise and self-consistent form.

\subsection{Gauss infinitesimal rigidity through the second derivative of the volume}
\label{subsec:IntroGaussVol}
Mixed volumes theory basically deals with derivatives of the volume bounded by $M$ with respect to infinitesimal deformations $\eta$ such that $d\eta(TM) \subset TM$. In the proof of the Gauss infinitesimal rigidity that can be extracted from \cite{Hil10}, one computes the second derivative $\Vol^{\bigdot\bigdot}$ under an isogauss deformation in two different ways. One computation yields zero, while the other one yields
$$
\Vol^{\bigdot\bigdot} = \int\limits_M 2h \det d\eta \, \darea,
$$
the right hand side of which coincides with the left hand side of \eqref{eqn:IntroHDet}. The operator $d\eta \colon TM \to TM$ turns out to be related to the derivative of $B^{-1}$ under the deformation $\eta$, which accounts for the properties $\det d\eta \le 0$ and $\det d\eta = 0 \Leftrightarrow d\eta = 0$.

\subsection{Darboux's wreath of 12 surfaces}
\label{subsec:IntroDarboux}
The arguments in all three Subsections \ref{subsec:IntroOurProof}, \ref{subsec:IntroBlaschke}, and \ref{subsec:IntroGaussVol} end by considering the same equation \eqref{eqn:IntroHDet}. As just explained, the last two arguments are identical. However, the argument sketched in Subsection \ref{subsec:IntroOurProof} arrives at \eqref{eqn:IntroHDet} in a different way. It makes use of another relation between the Gauss and metric infinitesimal rigidities.
\begin{quote}
The translation vector field of an isometric infinitesimal deformation of $M$ is an isogauss infinitesimal deformation of~$M^*$.
\end{quote}
Here $M^*$ is the polar dual of the surface $M$, and the \emph{translation vector field} $\tau$ is another vector field canonically associated with every isometric infinitesimal deformation $\xi$. The above correspondence together with that at the end of Subsection \ref{subsec:IntroBlaschkeOrigin} form a part of the \emph{Darboux wreath}, see Subsection \ref{subsec:DarbouxWreath} for more details. In fact, the correspondencies are local, so that $M$ may be an open disk embedded in $\R^3$ with non-vanishing Gauss curvature.

In order to compute $\HE^{\bigdot\bigdot}$ under the deformation $\xi$ of $M$, one performs the same integration by parts as when $\Vol^{\bigdot\bigdot}$ under the deformation $\tau$ of $M^*$ is computed. Besides, the deformation $\tau$ of $M^*$ is in some sense polar dual to the deformation $\xi$ of $M$. This suggests that there is a subtle relation between the Hilbert-Einstein functional of $P$ and the volume of $P^*$.

\subsection{Polar duality between the volume and the Hilbert-Einstein functional}
In spherical and hyperbolic geometries there is a straightforward relation between the Hilbert-Einstein functional and the volume of the dual, see Subsections \ref{subsec:HEPolDual} and \ref{subsec:HypDeSitter}. On the other hand, an analog in $\Sph^3$ and $\H^3$ of the Gauss infinitesimal rigidity in $\R^3$ is the infinitesimal rigidity with respect to the third fundamental form of the boundary. Since the third fundamental form equals the first fundamental form of the polar dual, the metric rigidity is directly related to the Gauss rigidity of the polar dual. As a result, variational approaches through $\HE$ and $\Vol$ to the respective kinds of rigidity are equivalent. For details, see Subsection \ref{subsec:ApprSphere}.

Note that the polar dual to a surface $M \subset \H^3$ is a surface in the de Sitter space, see Subsection \ref{subsec:HypDeSitter}.

\subsection{Minkowski and Weyl problems}
Gauss and metric infinitesimal rigidity can be interpreted as ``infinitesimal'' uniqueness in the Minkowski and Weyl problems. Besides, the Minkowski problem for polyhedra and for general convex surfaces can be solved by maximizing the functional $\Vol$ over some linear subspace in the space of all surfaces. A similar variational approach to the Weyl problem was suggested by Blaschke and Herglotz \cite{BH37}: every solution corresponds to a critical point of the Hilbert-Einstein functional on the space of all extensions of the given metric $g$ on the sphere to a metric $\wg$ on the ball, cf. Subsections \ref{subsec:IntroApproach} and \ref{subsec:IntroOurProof}. This approach was not realized, as the functional is neither concave nor convex, and there is no immediate reduction in the spirit of $\Vol$ and Minkowski problem above.

As a modification of Blaschke-Herglotz approach, we suggest to consider only extensions of the form $\wg_r$ described in Subsection \ref{subsec:IntroApproach}. Further details can be found in Subsection \ref{subsec:NewMinkWeyl}.

\subsection{General convex surfaces and other perspectives}
The Hilbert-Einstein functional has a discrete analog, and most of the constructions presented here carry over to polyhedral surfaces, see \cite{Izm11a}. A generalization to arbitrary convex surfaces suggests itself. In particular, it would be nice to find a similar proof of the infinitesimal rigidity of convex surfaces without flat pieces, a notoriously difficult theorem of Pogorelov \cite[Chapter IV]{Pog73}.

Another possible development is to extend our method to Einstein manifolds with boundary, cf. Subsection \ref{subsec:IntroKoiso}, in particular to hyperbolic manifolds with smooth convex boundary. (Infinitesimal rigidity of the latter is proved by Schlenker in \cite{Scl06} with a different method.) An extension to the case of non-regular boundary, as in the previous paragraph, would allow to prove the infinitesimal rigidity of convex cores, which could lead to a solution of the pleating lamination conjecture. See Section \ref{sec:Directions}.

\subsection{Acknowledgements}
I would like to thank Fran\c{c}ois Fillastre and Jean-Marc Schlenker for interesting discussions and useful remarks.

\section{Gauss rigidity of smooth convex surfaces}
\label{sec:GaussRigSm}
\subsection{Definitions and statement of the theorem}
\label{subsec:DefAndThm}
Let $M \subset \R^3$ be a smooth surface.

\begin{dfn}
An \emph{infinitesimal deformation} of a smooth surface $M \subset \R^3$ is a smooth vector field along $M$, i.~e. a smooth section of the vector bundle $T\R^3|_M \to M$.
\end{dfn}

By using the natural trivialization of the bundle $T\R^3|_M$, we can view an infinitesimal deformation of $M$ as a map
\begin{equation}
\label{eqn:EtaGauss}
\eta \colon M \to \R^3.
\end{equation}
The geodesic flow along an infinitesimal deformation $\eta$ produces a family of smooth maps
\begin{gather*}
\phi_t \colon M \to \R^3,\\
\phi_t = \exp(t\eta).
\end{gather*}
Equivalently, if $\eta$ is viewed as the map \eqref{eqn:EtaGauss}, then we have
\begin{equation}
\label{eqn:PhiT}
\phi_t(x) = x + t \eta(x).
\end{equation}
We restrict our attention to $t \in (-\epsilon, \epsilon)$ for some small positive $\epsilon$, so that all $\phi_t$ are smooth embeddings. Put
$$
M_t := \phi_t(M).
$$
Let $K_t \colon M_t \to \R$ be the Gauss curvature of $M_t$. Denote by $\dot K \colon M \to \R$ its derivative at $t=0$:
\begin{equation}
\label{eqn:DotK}
\dot K := \lim_{t \to 0} \frac{K_t \circ \phi_t - K}{t}.
\end{equation}

\begin{dfn}
\label{dfn:Isogauss}
An infinitesimal deformation $\eta$ is called \emph{isogauss}, if the following two conditions are fulfilled:
\begin{enumerate}
\item \label{it:IsogaussA} $T_xM$ is parallel to $T_{\phi_t(x)} M_t$ for all $x \in M$ and for all $t \in (-\epsilon, \epsilon)$;
\item $\dot K(x) = 0$ for all $x \in M$.
\end{enumerate}
\end{dfn}

In other words, the flow of an isogauss infinitesimal deformation preserves the Gauss map and preserves the Gauss curvature in the first order of $t$.

\begin{dfn}
\label{dfn:DeformTrivGauss}
An infinitesimal deformation $\eta$ of is called \emph{trivial}, if
$$
\eta(x) = a \quad \mbox{ for all } x \in M,
$$
for some $a \in \R^3$.
\end{dfn}

If $\eta$ is trivial, then the map $\phi_t$ is a parallel translation. Hence every trivial infinitesimal deformation is isogauss. If a surface has no isogauss infinitesimal deformations other than trivial ones, then we call it \emph{Gauss infinitesimally rigid}.

\begin{thm}
\label{thm:GaussRig}
Let $M \subset \R^3$ be a convex closed surface with everywhere positive Gauss curvature. Then $M$ is Gauss infinitesimally rigid
\end{thm}

We were unable to find a reference for Theorem \ref{thm:GaussRig}, but its statement and proof are implicitely contained in Hilbert's treatment \cite{Hil10} of Minkowski's mixed volumes theory, see also \cite{BF87} and \cite[Section 2.3]{Hoer07}. For the connection with the Minkowski problem see discussion in Subsection \ref{subsec:MWUniq}.

Theorem \ref{thm:GaussRig} holds for hypersurfaces in $\R^d$ as well, with $K$ interpreted as the determinant of the shape operator.

\subsection{Variation of the Gauss curvature}
\label{subsec:VarGaussGauss}
Let $\nu$ be a local unit normal field to $M$, and let $p$ be the position vector field on $M$. Similarly to $\eta$, these two vector fields can be viewed as smooth maps
\begin{gather*}
\nu \colon M \to \R^3,\\
p \colon M \to \R^3.
\end{gather*}
From this point of view, $p$ is the inclusion map $p(x) = x$.

Consider the differential
$$
d\nu \colon TM \to \R^3.
$$
The map $d\nu$ is obviously related to the covariant derivative
$$
\wnabla\nu \colon TM \to T|_M \R^3,
$$
namely, $d\nu(X)$ is the image of $\wnabla_X \nu$ under the map $T\R^3 \to \R^3$ which translates the starting point of a vector to the origin. We define the shape operator $B \colon TM \to TM$ as $B(X) = \wnabla_X\nu$ and write by an abuse of notation
$$
d\nu = B \colon TM \to TM.
$$
Similarly, the differential $dp$ interpreted as the covariant derivative is a map from $TM$ to $T|_M \R^3$. It is easy to see that
$$
dp = \id \colon TM \to TM.
$$
Finally, condition \eqref{it:IsogaussA} from Definition \ref{dfn:Isogauss} is clearly equivalent to $d\eta(X) \parallel TM$ for all $X$, so that we have yet another operator
$$
d\eta \colon TM \to TM.
$$

\begin{lem}
\label{lem:VarK}
Let $\eta$ be an infinitesimal deformation of a surface $M$ such that the condition \rm{(\ref{it:IsogaussA})} from Definition \ref{dfn:Isogauss} is satisfied. Then we have
\begin{equation}
\label{eqn:DotK TrEta}\dot K = - K \tr (d\eta).
\end{equation}
\end{lem}
\begin{proof}
Consider the surface $M_t$, see \eqref{eqn:PhiT}. Let $\nu_t$ be the unit normal field to $M_t$ (inducing the same co-orientation as $\nu$ on $M$), and $\eta_t$ be the position vector field. Condition (\ref{it:IsogaussA}) implies $\nu_t \circ \phi_t = \nu$. It follows that
$$
d\nu_t \circ d\phi_t = d\nu.
$$
By taking the determinant, we obtain
$$
(K_t \circ \phi_t) \cdot \det(\id + t d\eta) = K.
$$
After substituting in \eqref{eqn:DotK} and observing that
$$
\det(\id + t d\eta) = 1 + t \tr (d\eta) + t^2 \det(d\eta)
$$
we have 
$$
\dot K = \lim_{t \to 0} \frac{(K_t \circ \phi_t)(1 - \det(\id + t d\eta))}{t} = -K \tr(d\eta).
$$
The lemma is proved.
\end{proof}

\subsection{Proof of Theorem \ref{thm:GaussRig}}
\label{subsec:ProofGaussRig}
The proof is based on two lemmas.

\begin{lem}
\label{lem:DetDEtaNeg}
Let $\eta$ be an isogauss infinitesimal deformation of $M$. Assume that $M$ has a positive Gauss curvature at $x$. Then at this point $x$ we have
\begin{enumerate}
\item \label{it:DetNeg} $\det(d\eta) \le 0$;
\item \label{it:DetZero} the equality $\det(d\eta) = 0$ holds only if $d\eta = 0$.
\end{enumerate}
\end{lem}
\begin{proof}
On the four-dimensional vector space $\End(T_xM)$, consider the quadratic form $\det$. The associated symmetric bilinear form is called the \emph{mixed determinant}, see Subsection \ref{subsec:MixDetOp} for more details. As $\det(\id, A) = \tr A$ for all $A$, equation \eqref{eqn:DotK TrEta} implies
\begin{equation}
\label{eqn:DetIdEta}
\det(\id, d\eta) = 0,
\end{equation}
because we have $\dot K = 0$ and $K \ne 0$. By differentiating the identity
$$
\langle \nu, d\eta \rangle = 0
$$
and using the fact that the connection $\wnabla$ is flat, we also obtain
\begin{equation}
\label{eqn:DetJNuEta}
\det(JB, d\eta) = 0,
\end{equation}
where $J \colon T_xM \to T_xM$ is the rotation by $\frac{\pi}{2}$. Thus the vector $d\eta \in \End(T_xM)$ is orthogonal to the vectors $\id$ and $JB$ with respect to the form $\det(\cdot \,, \cdot)$.

On the other hand, $\det$ takes positive values on both $\id$ and $JB$; the latter value is $K > 0$ by our assumption. As the signature of $\det$ is $(+, +, -, -)$ and vectors $\id$ and $JB$ are linearly independent, this implies that $d\eta$ lies in a two-dimensional subspace of $\End(T_xM)$ on which $\det$ is negative definite. This implies both statements of the lemma.
\end{proof}

\begin{dfn}
\label{dfn:SuppFunc}
Let $M \subset \R^3$ be an embedded orientable surface with a chosen unit normal field $\nu$. The \emph{support function}
$$
h \colon M \to \R
$$
sends every point $x$ to the (signed) distance of the tangent plane $T_xM$ from the coordinate origin in $\R^3$. The distance is considered positive if the normal $\nu$ points away from the origin.
\end{dfn}

Clearly, we have $h(x) = \langle \nu(x), x \rangle$. In terms of the position vector field~$p$ this can be written as
$$
h = \langle \nu, p \rangle.
$$

\begin{lem}
\label{lem:IntZero}
For every isogauss infinitesimal deformation $\eta$ of a closed surface $M \subset \R^3$ we have
\begin{equation}
\label{eqn:IntZero}
\int\limits_M 2 h \det(d\eta)\, \darea_M = 0,
\end{equation}
where $h$ is the support function of $M$.
\end{lem}
\begin{proof}
The proof uses partial integration. We will present it in terms of vector-valued differential forms, see Section~\ref{sec:DetForms} for more details.

View a map $\eta \colon M \to \R^3$ as an $\R^3$--valued differential 0--form. The differential $d\eta$ is naturally an $\R^3$--valued differential 1--form. Wedge product of vector-valued forms takes values in the tensor power of $\R^3$. For example, $p \wedge d\eta \wedge d\eta$ is a differential 2--form on $M$ with values in $\R^3 \otimes \R^3 \otimes \R^3$. Consider the linear map
$$
\dvol \colon \R^3 \otimes \R^3 \otimes \R^3 \to \R,
$$
sending a triple of vectors to the signed volume of the parallelepiped spanned by them. By applying $\dvol$ to the coefficient of the form $p \wedge d\eta \wedge d\eta$, we obtain a real-valued 2--form
$$
\dvol(p \wedge d\eta \wedge d\eta) \in \Omega^2(M).
$$
At the same time, $d\eta$ can be viewed as a $TM$--valued 1--form, see Subsection \ref{subsec:VarGaussGauss}, so that $\darea_M(d\eta \wedge d\eta)$ is also a 2--form on $M$. We have
\begin{equation}
\label{eqn:PEtaEta}
\dvol(p \wedge d\eta \wedge d\eta) = h \darea_M (d\eta \wedge d\eta) = 2h \det(d\eta) \darea_M.
\end{equation}

Due to the Leibniz rule and to $d(d\eta) = 0$ (see Subsections \ref{subsec:BundleVal} and \ref{subsec:VecVal}) we have
\begin{equation}
\label{eqn:DPEtaEta}
d(\dvol(p \wedge d\eta \wedge d\eta)) = \dvol(dp \wedge \eta \wedge d\eta) + \dvol(p \wedge d\eta \wedge d\eta).
\end{equation}
Similarly to \eqref{eqn:PEtaEta},
$$
\dvol(dp \wedge \eta \wedge d\eta) = - \langle \nu, \eta \rangle \darea_M(dp \wedge d\eta) = -2 \langle \nu, \eta \rangle \det(dp, d\eta) \darea_M.
$$
As $dp = \id$, equation \eqref{eqn:DetIdEta} (which holds because $\eta$ is isogauss) implies
$$
\dvol(dp \wedge \eta \wedge d\eta) = 0.
$$
Thus, integrating \eqref{eqn:DPEtaEta} and applying the Stokes formula yields
$$
\int\limits_M \dvol(p \wedge d\eta \wedge d\eta) = 0.
$$
Together with \eqref{eqn:PEtaEta} this implies the statement of the lemma.
\end{proof}

\begin{proof}[Proof of Theorem \ref{thm:GaussRig}]
Let $M$ be a closed surface with everywhere positive Gauss curvature, and let $\eta$ be an isogauss infinitesimal deformation of $M$. We may assume that the coordinate origin of $\R^3$ lies in the region bounded by $M$, which implies $h > 0$. Together with Lemma \ref{lem:DetDEtaNeg} \eqref{it:DetNeg} this yields
$$
\int\limits_M 2h \det(d\eta) \, \darea_M \le 0.
$$
On the other hand, by Lemma \ref{lem:IntZero} equation \eqref{eqn:IntZero} holds. Thus we must have $\det(d\eta) = 0$ everywhere on $M$. By Lemma \ref{lem:DetDEtaNeg} \eqref{it:DetNeg} this implies $d\eta = 0$. Hence the map $\eta \colon M \to \R^3$ is constant i.~e. is a trivial deformation. The theorem is proved.
\end{proof}

\begin{rem}
Pull the shape operator $B_t \colon TM_t \to TM_t$ back to $TM$ by identifying $TM_t$ with $TM$ through parallel translation. Then it is easy to see that
$$
B_t = (\id + t d\eta)^{-1} \circ B,
$$
which implies
$$
\dot B = - d\eta \circ B.
$$
This equation leads to alternative proofs of equations \eqref{eqn:DotK TrEta} and \eqref{eqn:DetJNuEta}:
\begin{equation*}
\begin{split}
\dot K &=  (\det B)^\bigdot = 2 \det(\dot B, B) = -2 \det(d\eta \circ B, B)\\
&= -2 \det(d\eta, \id) \cdot \det B = - \tr (d\eta) \cdot K
\end{split}
\end{equation*}
\begin{equation*}
\begin{split}
\det(JB, d\eta) &= \det(JB, - \dot B B^{-1}) = \det(J, - B^{-1} \dot B B^{-1}) \cdot \det B\\
&= \det(J, (B^{-1})^\bigdot) \cdot \det B = 0,
\end{split}
\end{equation*}
where the last equation holds because $\det(J, A) = 0$ for every self-adjoint operator $A$.
\end{rem}

\subsection{Second derivative of the volume}
\label{subsec:VolDer}
Here we give a geometric interpretation of the integral $\int_M h \det(d\eta)\, \darea_M$ that plays a crucial role in the proof of Theorem \ref{thm:GaussRig}.

\begin{lem}
\label{lem:VolP}
Let $P \subset \R^3$ be the body bounded by a closed surface $M$, and let $p \colon M \to \R^3$ be the position vector field along $M$. Then we have
\begin{equation}
\label{eqn:VolP}
\Vol(P) = \frac16 \int\limits_M \dvol(p \wedge dp \wedge dp).
\end{equation}
\end{lem}
\begin{proof}
We have
$$
\dvol(dp \wedge dp \wedge dp) = 6
$$
(cf. Lemma \ref{lem:DvolDet}). Thus
$$
\Vol(P) = \frac16 \int\limits_P \dvol(dp \wedge dp \wedge dp) = \frac16 \int\limits_M \dvol(p \wedge dp \wedge dp),
$$
where we applied the Stokes theorem in the last equality. The lemma is proved.
\end{proof}

\begin{rem}
In a more classical notation, formula \eqref{eqn:VolP} reads as
$$
\Vol(P) = \frac13 \int\limits_M h \, \darea_M
$$
and can be obtained similarly by integrating the function $\Delta(\|x\|^2) = 6$ over $P$ and applying the Stokes theorem.
\end{rem}

Let $\eta$ be an infinitesimal deformation of $M$ satisfying condition \eqref{it:IsogaussA} from Definition \ref{dfn:Isogauss}. Let $M_t$ be the surface defined in Subsection \ref{subsec:DefAndThm}, and denote by $P_t$ the body bounded by $M_t$. Denote further
$$
\Vol^\bigdot = \left. \frac{d}{dt} \right|_{t=0} \Vol(P_t), \qquad
\Vol^{\bigdot\bigdot} := \left. \frac{d^2}{dt^2} \right|_{t=0} \Vol(P_t).
$$

\begin{lem}
\label{lem:VolDer}
We have
\begin{gather}
\label{eqn:VolDot}
\Vol^\bigdot = \frac12 \int\limits_M \dvol(\eta \wedge dp \wedge dp) = \frac12 \int\limits_M \dvol(p \wedge d\eta \wedge dp),\\
\label{eqn:VolDDot}
\Vol^{\bigdot\bigdot} = \int\limits_M \dvol(\eta \wedge d\eta \wedge dp) = \int\limits_M \dvol(p \wedge d\eta \wedge d\eta).
\end{gather}
\end{lem}

\begin{proof}
Pull back the position vector field $p_t \colon M_t \to \R^3$ to $M$ by the map $\phi_t$. Retaining the same notation, we have
$$
p_t = p + t\eta.
$$
Lemma \ref{lem:VolP} implies
$$
\Vol(P_t) = \frac16 \int\limits_M \dvol(p_t \wedge dp_t \wedge dp_t).
$$
By differentiating and using Lemma \ref{lem:SymmDet} we obtain
$$
\Vol^\bigdot = \frac16 \int\limits_M \dvol(\eta \wedge dp \wedge dp) + \frac13 \int\limits_M \dvol(p \wedge d\eta \wedge dp).
$$
Integration by parts shows that the integrals on the right hand side are equal. This implies equations \eqref{eqn:VolDot}.

Equations \eqref{eqn:VolDDot} are proved in a similar way. Formula for integration by parts appeared already in \eqref{eqn:DPEtaEta}.
\end{proof}

Equations \eqref{eqn:PEtaEta} and \eqref{eqn:VolDDot} imply
$$
\Vol^{\bigdot\bigdot} = \int_M 2h \det(d\eta)\, \darea_M.
$$
The proof of Theorem \ref{thm:GaussRig} given in Subsection \ref{subsec:ProofGaussRig} consists in computing the second derivative of $\Vol$ in two different ways --- the first and the second integral in \eqref{eqn:VolDDot} --- and then showing that the first integral vanishes (Lemma \ref{lem:IntZero}) while the second one is non-positive and can vanish only if $\eta$ is trivial (Lemma \ref{lem:DetDEtaNeg}).

\subsection{Hessian of the support function}
\label{subsec:HessSupp}
A more traditional way to present the above proof is in terms of the support function.

First, we need to change the setup. Instead of considering a closed convex surface $M \subset \R^3$, we consider a smooth embedding
$$
\phi \colon \Sph^2 \to \R^3.
$$
An infinitesimal deformation
$$
\eta \colon \Sph^2 \to \R^3
$$
determines a family of embeddings
$$
\phi_t = \phi + t\eta.
$$
Finally, the position vector field $p$ coincides with the map $\phi$:
$$
p = \phi.
$$
If $M = \phi(\Sph^2)$ has everywhere positive Gauss curvature, then the Gauss map from $M$ to $\Sph^2$ is one-to-one. So we can make the following important assumption: \emph{$\phi$ is the inverse of the Gauss map}. In other words, we view $\Sph^2$ as a unit sphere centered at $0 \in \R^3$, and the map $\phi$ sends every $x \in \Sph^2$ to a point on $M$ with outward unit normal $x$.

The support function is now also being viewed as a function on $\Sph^2$, namely
\begin{equation}
\label{eqn:HOnSph}
h = \langle \iota, p \rangle,
\end{equation}
where $\iota \colon \Sph^2 \to \R^3$ is the inclusion map.

\begin{lem}
The position vector of a Gauss image parametrized surface is given by
\begin{equation}
\label{eqn:PNablaH}
p = h \iota + \nabla h,
\end{equation}
where $\nabla$ is the Levi-Civita connection on $\Sph^2$.
\end{lem}
\begin{proof}
By differentiating \eqref{eqn:HOnSph} in the direction of a vector $X \in T_x \Sph^2$, we obtain
$$
X(h) = \left\langle \wnabla_X \iota, p(x) \right\rangle + \left\langle \iota(x), \wnabla_X p \right\rangle = \langle X, p(x) \rangle + \langle x, d\phi(X) \rangle.
$$
As $d\phi(X) \in T_\phi(x)M \parallel T_x\Sph^2$, the last summand vanishes and we have
$$
\langle X, \nabla h \rangle = \langle X, p(x) \rangle.
$$
It follows that $\nabla h$ is the orthogonal projection of $p(x)$ to $T_x \Sph^2$, that is
$$
\nabla h = p(x) - \langle x, p(x) \rangle x = p(x) - h(x) x,
$$
and the lemma is proved.
\end{proof}

The linear map $dp \colon T\Sph^2 \to T\R^3$ is related to the shape operator. Namely, as $p = \phi$ is the inverse of the Gauss map, $dp$ postcomposed with the parallel translation of $T_\phi(x)M$ to $T_x\Sph^2$ is the pullback of $B^{-1}$ to $\Sph^2$:
\begin{equation}
\label{eqn:DpB}
dp = \phi^* (B^{-1}).
\end{equation}

\begin{lem}
Let $\phi \colon \Sph^2 \to M$ be the inverse of the Gauss map. Then the pullback of the inverse of the shape operator expresses in terms of the support function $h \colon \Sph^2 \to \R$ as
\begin{equation}
\label{eqn:BHessH}
\phi^* (B^{-1}) = h \, \id + \Hess h,
\end{equation}
where $\Hess h \colon X \mapsto \nabla_X(\nabla h)$ is the $(1,1)$--Hessian of function $h$.
\end{lem}
\begin{proof}
By differentiating \eqref{eqn:PNablaH} and using \eqref{eqn:DpB} we obtain at every point $x \in \Sph^2$ the equation
$$
\phi^* (B^{-1})(X) = X(h) x + h X + \wnabla_X (\nabla h).
$$
As the left hand side lies in $T_x\Sph^2$, we have
\begin{eqnarray*}
X(h) x + h X + \wnabla_X (\nabla h) & = & \top(X(h) x + h X + \wnabla_X (\nabla h))\\
& = & h X + \top(\wnabla_X (\nabla h)) = h X + \nabla_X (\nabla h),
\end{eqnarray*}
where $\top \colon T_x \R^3 \to T_x \Sph^2$ is the orthogonal projection. This proves the lemma.
\end{proof}

\begin{lem}
The volume of the body $P$ bounded by $M$ expresses in terms of the support function as
$$
\Vol(P) = \frac13 \int\limits_{\Sph^2} h \det(h \, \id + \Hess h) \, \darea_{\Sph^2}.
$$
\end{lem}
\begin{proof}
This follows from
$$
\Vol(P) = \frac13 \int\limits_M h \circ \phi^{-1} \, \darea_M = \frac13 \int\limits_{\Sph^2} h \det(B^{-1}) \, \darea_{\Sph^2}
$$
and from equation \eqref{eqn:BHessH}.
\end{proof}

By equation \eqref{eqn:PNablaH}, the position vector field and the support function are related in a linear manner. Therefore every infinitesimal deformation can be expressed in terms of the support function by the same formula:
$$
\eta = k \iota + \nabla k,
$$
where $k \colon \Sph^2 \to \R$ is an arbitrary smooth function. It follows that formulas \eqref{eqn:VolDDot} can be rewritten as
\begin{equation}
\label{eqn:VolDDotBInv}
\Vol^{\bigdot\bigdot} = \int\limits_{\Sph^2} k \det(B^{-1}, (B^{-1})^\bigdot) \, \darea_{\Sph^2} = \int\limits_{\Sph^2} h \det((B^{-1})^\bigdot) \, \darea_{\Sph^2},
\end{equation}
or
\begin{eqnarray*}
\Vol^{\bigdot\bigdot} & = & \int\limits_{\Sph^2} k \det(k \, \id + \Hess k, h \, \id + \Hess h) \, \darea_{\Sph^2}\\
& = & \int\limits_{\Sph^2} h \det(h \, \id + \Hess h) \, \darea_{\Sph^2}.
\end{eqnarray*}
The proof of rigidity goes similarly to Subsection \ref{subsec:ProofGaussRig}. The assumption $\dot K = 0$ implies
$$
\det(B^{-1}, (B^{-1})^\bigdot) = \frac12 (\det (B^{-1}))^\bigdot = 0,
$$
so that the first integrand in \eqref{eqn:VolDDotBInv} vanishes. Non-positivity of the second integrand is proved simpler than in Lemma \ref{lem:DetDEtaNeg}: as $B^{-1}$ is self-adjoint, it suffices to prove its $\det$--orthogonality to \emph{one} operator with positive determinant, and we can just use for this the last equation (cf. Corollary \ref{cor:DetDotBNeg}).

\begin{rem}
The Hessian of the support function was used by Hilbert \cite{Hil10} to provide a basis of mixed volumes theory for bodies with smooth boundary established earlier by Minkowski for general convex bodies. The argument outlined in this section serves as a lemma in a proof of the Alexandrov-Fenchel inequality. A recent generalization of the latter is obtained in \cite{GMTZ10}.
\end{rem}

\section{Metric rigidity of smooth convex surfaces}
\subsection{The Liebmann-Blaschke theorem}
\label{subsec:BLTheorem}
Let $M \subset \R^3$ be a smooth surface, and let
$$
\xi \colon M \to \R^3,
$$
be a vector field along $M$. Similarly to Subsection \ref{subsec:DefAndThm}, consider for small $t$ the family of surfaces
\begin{gather*}
M_t = \phi_t(M),\\
\phi_t(x) = x + t\xi(x).
\end{gather*}
Let $g_t$ be the induced Riemannian metric on $M_t$. We write $g_0 = g$. Put
$$
\dot g = \lim_{t \to 0} \frac{\phi_t^*g_t - g}{t}.
$$

\begin{dfn}
An infinitesimal deformation $\xi$ of a surface $M \subset \R^3$ is called \emph{isometric}, if $\dot g = 0$.
\end{dfn}

\begin{dfn}
\label{dfn:TrivXi}
An infinitesimal deformation $\xi$ of a surface $M \subset \R^3$ is called \emph{trivial}, if $\xi$ is the restriction of an infinitesimal isometry (Killing vector field) of $\R^3$:
$$
\xi(x) = a \times x + b
$$
for some $a, b \in \R^3$, with $\times$ meaning the cross product of vectors in $\R^3$.
\end{dfn}

As the Lie derivative of $\can_{\R^3}$ along a Killing vector field vanishes, every trivial infinitesimal deformation is isometric. If a surface has no isometric infinitesimal deformations other than trivial ones, then it is called (metrically) infinitesimally rigid.

\begin{thm}[Liebmann-Blaschke-Weyl]
\label{thm:InfRigSm}
Let $M \subset \R^3$ be a convex closed surface with everywhere positive Gauss curvature. Then $M$ is infinitesimally rigid.
\end{thm}

This theorem was proved by Liebmann in \cite{Lieb99} under the analyticity assumption, and later by Blaschke and Weyl for $C^3$--surfaces. A modern version of Blaschke-Weyl's proof is given in Subsection \ref{subsec:BlaschkeProof}, the background of this proof is explained in Subsection \ref{subsec:ShearBend}.

\subsection{Rotation and translation vector fields of an isometric infinitesimal deformation}
\label{subsec:RotTranslFields}
Recall from Subsection \ref{subsec:VarGaussGauss} that we denote by
$$
p \colon M \to \R^3
$$
the position vector field $p(x) = x$, and that for any vector field, say $\xi$, we view its differential (or covariant derivative) as a map
$$
d\xi \colon TM \to \R^3.
$$
In particular, $dp(X) = X$ for every $X \in TM$.

The following lemma is due to Darboux \cite{Dar96}.

\begin{lem}
\label{lem:EtaTau}
Let $\xi$ be an isometric infinitesimal deformation of a surface $M$. Then there exists a unique pair $(\eta, \tau)$ of vector fields along $M$ such that
$$
\xi = \eta \times p + \tau
$$
and moreover
\begin{gather}
d\xi = \eta \times dp \label{eqn:DXi},\\
d\tau = p \times d\eta. \label{eqn:DTau}
\end{gather}
\end{lem}
\begin{proof}
The condition $\dot g = 0$ is easily seen to be equivalent to
$$
\langle d\xi(X), Y \rangle + \langle X, d\xi(Y) \rangle = 0
$$
for all $X, Y \in TM$. This implies that $d\xi \colon TM \to \R^3$ has a unique extension to a skew-symmetric operator $A \colon \R^3 \to \R^3$. As every skew-symmetric operator in $\R^3$ is the cross product with a fixed vector, this defines a vector field $\eta \colon M \to \R^3$ satisfying condition \eqref{eqn:DXi}. If we put $\tau = \xi - \eta \times p$, then the condition \eqref{eqn:DTau} is automatically fulfilled. The lemma is proved.
\end{proof}

Intuitively: an isometric infinitesimal deformation $\xi$ moves every tangent plane $T_xM$ as a rigid body. Thus there exists a first-order approximation of $\xi$ in a neigborhood of every point $x_0$ by a Killing vector field $x \mapsto \eta(x_0) \times x + \tau(x_0)$.

\begin{dfn}
Vector fields $\eta$ and $\tau$ are called the \emph{rotation} and \emph{translation vector fields} of an isometric infinitesimal deformation~$\xi$.
\end{dfn}

Note that the rotation field is invariant under translations of surface $M$ while the translation field is not. Namely, if $M$ is translated by a vector $a \in \R^3$ and $\xi$ kept unchanged, then the corresponding vector fields are
$$
p' = p + a, \quad \eta' = \eta, \quad \tau' = \tau + \eta \times a.
$$

\begin{lem}
\label{lem:DEtaTM}
We have $d\eta(X) \in TM$ for all $X \in TM$.
\end{lem}
\begin{proof}
This can be proved by taking the differential of \eqref{eqn:DXi} (or computing second covariant derivatives $\wnabla_X \wnabla_Y \xi$ and $\wnabla_Y \wnabla_X \xi$).
\end{proof}

\begin{lem}
\label{lem:ConstRot}
An isometric infinitesimal deformation is trivial if and only if its rotation vector field is constant.
\end{lem}
\begin{proof}
The rotation vector field of a trivial isometric infinitesimal deformation is constant by Definition \ref{dfn:TrivXi} and Lemma \ref{lem:EtaTau}.

In the opposite direction, if $\eta = \const$, then $d\eta = 0$, and by \eqref{eqn:DTau} $d\tau = 0$. Hence $\tau = \const$. By definition, if both $\eta$ and $\tau$ are constant, then $\xi$ is trivial.
\end{proof}

\subsection{Variation of the shape operator}
\label{subsec:VarShape}
Let $\nu \colon M \to \R^3$ be a local unit normal field to $M$, and let $\nu_t \colon M_t \to \R^3$ be the corresponding local unit normal field on $M_t$. Consider the shape operators $B = d\nu$ and
$$
B_t = d\nu_t \colon TM_t \to TM_t.
$$
Denote
$$
\dot B = \lim_{t \to 0} \frac{\phi_t^* B_t - B}{t}.
$$

\begin{lem}
\label{lem:BSelfAdj}
For every isometric infinitesimal deformation $\xi$, the operator $\dot B \colon TM \to TM$ is self-adjoint.
\end{lem}
\begin{proof}
The self-adjointness of $B_t$ with respect to $g_t$ implies
$$
(\phi_t^*g_t)((\phi_t^* B_t)(X), Y) = (\phi_t^*g_t)(X, (\phi_t^* B_t)(Y))
$$
for all $X, Y \in TM$. By differentiating with respect to $t$, we obtain
$$
\dot g(B(X), Y) + g(\dot B(X), Y) = \dot g(X, B(Y)) + g(X, \dot B(Y)).
$$
As $\xi$ is isometric, we have $\dot g = 0$. This implies the lemma.
\end{proof}

The next lemma relates operators $d\eta$ and $d\tau$ with $\dot B$.

\begin{lem}
\label{lem:DEtaBDot}
Let $\xi$ be an isometric infinitesimal deformation of a surface $M \subset \R^3$, and let $\eta$ and $\tau$ be its rotation and translation vector fields. Then we have
\begin{equation}
\label{eqn:DEtaBDot}
d\eta = J \dot B,
\end{equation}
where $J \colon TM \to TM$ is rotation by $\pi/2$. We also have
\begin{equation}
\label{eqn:DTauBDot}
\top \circ d\tau = - h \dot B,
\end{equation}
where $\top \colon \R^3 \to TM$ is an orthogonal projection, and $h \colon M \to \R$ is the support function, see Definition \ref{dfn:SuppFunc}.
\end{lem}
\begin{proof}
Denote
$$
\dot \nu := \lim_{t \to 0} \frac{\nu_t \circ \phi_t - \nu}{t}.
$$
By differentiating equations $\|\nu_t\|^2 = 1$ and $\langle \nu_t, d\phi_t(X) \rangle = 0$ and using \eqref{eqn:DXi}, we obtain
$$
\dot \nu = \eta \times \nu.
$$
(This is intuitively clear: if a tangent plane to $M$ is rotated by $\eta$, then the normal is also rotated by $\eta$.) Thus we have
\begin{equation}
\label{eqn:DNu1}
d \dot\nu = d\eta \times \nu + \eta \times d\nu = -J \circ d\eta + \eta \times B,
\end{equation}
where we used that $d\eta(TM) \subset TM$, see Lemma \ref{lem:DEtaTM}.

On the other hand,
$$
d(\nu_t \circ \phi_t) = d\nu_t \circ d\phi_t = d\phi_t \circ \phi_t^* B_t.
$$
Differentiating with respect to $t$ yields
\begin{equation}
\label{eqn:DNu2}
d \dot\nu = d\xi \circ B + dp \circ \dot B = \eta \times B + \dot B,
\end{equation}
compare \cite[Equation (0)]{Sou99}. Equating the right hand sides of \eqref{eqn:DNu1} and \eqref{eqn:DNu2} yields equation \eqref{eqn:DEtaBDot}.

Now let us prove Equation \eqref{eqn:DTauBDot}. We have
$$
d\tau = p \times d\eta = (h\nu + \top p) \times d\eta = h\nu \times d\eta + c\nu
$$
for some $c \in \R$, because $d\eta(X) \in TM$. Hence, due to \eqref{eqn:DEtaBDot}
$$
\top \circ d\tau = h\nu \times d\eta = hJ \circ d\eta = h J^2 \dot B = - h \dot B.
$$
The lemma is proved.
\end{proof}

\subsection{Blaschke's proof of Theorem \ref{thm:InfRigSm}}
\label{subsec:BlaschkeProof}
Here we give a coordinate-free version of the proof from \cite{Bla21} (see also \cite[Chapter 12]{SpiV}). The main novelty is the interpretation of the operator $d\eta \colon TM \to TM$ through the variation of the shape operator, see Lemma \ref{lem:DEtaBDot}.

\begin{proof}[Blaschke's proof of Theorem \ref{thm:InfRigSm}]
Let $M \subset \R^3$ be a closed surface with positive Gauss curvature, and let $\xi$ be an isometric infinitesimal deformation of $M$. Let $\eta$ be the rotation vector field of $\xi$, see Subsection \ref{subsec:RotTranslFields}. By Lemma \ref{lem:DEtaBDot}, we have
$$
d\eta = J \dot B \colon TM \to TM.
$$
This implies two remarkable properties of the operator $d\eta$: first,
\begin{equation}
\label{eqn:TrEta}
\tr(d\eta) = 0,
\end{equation}
as $\tr(JA) = 0$ for every self-adjoint operator $A$, and $\dot B$ is self-adjoint by Lemma \ref{lem:BSelfAdj}, and second
\begin{equation}
\label{eqn:DetEta}
\det(d\eta) \le 0, \quad \mbox{ while } \det(d\eta) = 0 \mbox{ only if } d\eta = 0,
\end{equation}
which holds by Corollary \ref{cor:DetDotBNeg} as $\det B = K > 0$ by assumption on $M$ and $(\det B)^\bigdot = \dot K = 0$ because $\xi$ is isometric.

Recall that $p \colon M \to \R^3$ denotes the position vector field so that we have $dp = \id \colon TM \to TM$. Consider the differential 1--form
\begin{equation}
\label{eqn:BlaschkeForm}
\dvol(p \wedge \eta \wedge d\eta) \in \Omega^1(M),
\end{equation}
see Subsections \ref{subsec:MixDetVVal} --- \ref{subsec:VecVal} or a brief explanation in the proof of Lemma \ref{lem:IntZero}. We have
$$
d(\dvol(p \wedge \eta \wedge d\eta)) = \dvol(dp \wedge \eta \wedge d\eta) + \dvol(p \wedge d\eta \wedge d\eta).
$$
The first summand on the right hand side vanishes due to
$$
dp \wedge d\eta = \det(\id, d\eta) \darea_M = \frac12 \tr(d\eta) \darea_M = 0,
$$
where we used \eqref{eqn:TrEta} in the last step. The integrand in the second summand equals $2h \det(d\eta)$, where $h$ is the support function of the surface $M$. Thus Stokes' theorem implies
\begin{equation}
\label{eqn:IntBlaschke0}
\int\limits_{M} 2h \det(d\eta)\, \darea_M = 0.
\end{equation}
Without loss of generality we may assume that the origin of $\R^3$ lies in the region bounded by $M$, so that $h>0$. Together with \eqref{eqn:DetEta}, this implies $d\eta=0$. Thus $\eta$ is constant, and by Lemma \ref{lem:ConstRot} the deformation $\xi$ is trivial.
\end{proof}

Subsection \ref{subsec:ShearBend} tells about the origin of this proof. We'll now give another proof that looks more complicated at first but allows a variational interpretation similar to that given in Subsection \ref{subsec:VolDer} to the proof of Theorem~\ref{thm:GaussRig}.

\subsection{Another proof of Theorem \ref{thm:InfRigSm}}
\label{subsec:OtherProof}
Let $\nu$ be the field of outward unit normals to $M$. Introduce a new vector field along $M$
$$
\psi = \frac{\nu}{h}.
$$
Consider the differential 1--form
$$
\dvol(\psi \wedge \tau \wedge d\tau) \in \Omega^1(M)
$$
and its differential
$$
d(\dvol(\psi \wedge \tau \wedge d\tau)) = \dvol(d\psi \wedge \tau \wedge d\tau) + \dvol(\psi \wedge d\tau \wedge d\tau).
$$
By Stokes' theorem,
\begin{equation}
\label{eqn:StokesMine}
\int\limits_M \dvol(\tau \wedge d\psi \wedge d\tau) = \int\limits_M \dvol(\psi \wedge d\tau \wedge d\tau).
\end{equation}

\begin{lem}
\label{lem:Tau1}
We have
$$
\dvol(\tau \wedge d\psi \wedge d\tau) = 0.
$$
\end{lem}
\begin{proof}
First of all, let us show that
$$
\langle d\tau, p \rangle = 0 = \langle d\psi, p \rangle.
$$
The first equality follows from $d\tau = p \times d\eta$. For the second one, compute
$$
\langle d\psi(X), p \rangle = X \langle \psi, p \rangle - \langle \psi, X \rangle.
$$
The right hand side vanishes as $\psi \perp X$ and $\langle \psi, p \rangle = 1$. Thus the left hand side vanishes too, and the claim is proved.

It follows that
\begin{equation}
\label{eqn:DVolTauPsiPsi}
\dvol(\tau \wedge d\psi \wedge d\tau) = \frac{\langle \tau, p \rangle}{\|p\|} \, \darea_{p^\perp}(d\psi \wedge d\tau),
\end{equation}
where $\darea_{p^\perp}$ is the area form in the orthogonal complement to vector~$p$.

Substitute $d\tau = p \times d\eta$. A simple computation shows that
$$
\darea_{p^\perp}(\psi \wedge (p \times d\eta))(X,Y) = \|p\| \left( \langle \psi(X), d\eta(Y) \rangle - \langle \psi(Y), d\eta(X) \rangle \right).
$$
By using \eqref{eqn:DEtaBDot}, we compute
\begin{equation*}
\begin{split}
\langle d\psi(X), d\eta(Y) \rangle & = \langle h^{-1} B(X) + X(h^{-1}) \nu, J \dot B(Y) \rangle\\
& = h^{-1} \langle B(X), J \dot B(Y) \rangle = - h^{-1} \darea_M(B(X), \dot B(Y)).
\end{split}
\end{equation*}
It follows that
\begin{equation}
\label{eqn:DPsiDTau}
\darea_{p^\perp}(d\psi \wedge d\tau) = -\frac{\|p\|}{h} \darea_M(B, \dot B) = 0,
\end{equation}
because $2\det(B, \dot B) = (\det B)^\bigdot = \dot K = 0$ for an isometric infinitesimal deformation. Substitution in \eqref{eqn:DVolTauPsiPsi} proves the lemma.
\end{proof}

\begin{lem}
\label{lem:Tau2}
We have
$$
\dvol(h^{-1}\nu \wedge d\tau \wedge d\tau) = 2h \det \dot B \, \darea_M.
$$
\end{lem}
\begin{proof}
Recall that $\top \colon \R^3 \to TM$ denotes the orthogonal projection (thus $\top$ sends a vector in $T\R^3|_M$ to its tangential component). We have
\begin{equation*}
\begin{split}
\dvol(h^{-1}\nu \wedge d\tau \wedge d\tau) &= h^{-1} \darea_M(\top \circ d\tau, \top \circ d\tau)\\
&= 2 h^{-1} \det(\top \circ d\tau) = 2 h^{-1} \det(- h \dot B) = 2 h \det \dot B,
\end{split}
\end{equation*}
where we used equation \eqref{eqn:DTauBDot}. The lemma is proved.
\end{proof}

By substituting the results of Lemmas \ref{lem:Tau1} and \ref{lem:Tau2} in \eqref{eqn:StokesMine}, we obtain
\begin{equation}
\label{eqn:IntHDotB}
\int\limits_M 2 h \det \dot B \, \darea_M = 0.
\end{equation}
The end of the proof is the same as in Subsection \ref{subsec:BlaschkeProof}: we may assume $h > 0$, at the same time by Corollary \ref{cor:DetDotBNeg} we have $\det \dot B \le 0$; thus we have $\det \dot B = 0$ which again by Corollary \ref{cor:DetDotBNeg} implies $\dot B = 0$. As $d\eta = J \dot B$, it follows that the rotation vector field is constant, thus by Lemma \ref{lem:ConstRot} the deformation $\xi$ is trivial.

\subsection{Hessian of the squared distance function}
\label{subsec:HessSquaredDist}
Here we rewrite the proof from the previous subsection in other terms. For a surface $M \subset \R^3$, consider the function
\begin{gather*}
f \colon M \to \R,\\
f(x) = \frac{\|x\|^2}2.
\end{gather*}

\begin{lem}
\label{lem:NablaFEucl}
We have
$$
\nabla f = p - h\nu,
$$
where $p \colon M \to \R^3$ is the position vector field, and $h \colon M \to \R$ is the support function relative to the unit normal field $\nu$.
\end{lem}
\begin{proof}
Clearly, $\wnabla f = p$. Thus we have
$$
\nabla f = \top p = p - \langle \nu, p \rangle \nu = p - h\nu.
$$
\end{proof}

\begin{lem}
We have
\begin{equation}
\label{eqn:HessFEucl}
\Hess f = \id - hB,
\end{equation}
where $\Hess f \colon TM \to TM$, $X \mapsto \nabla_X \nabla f$ is the $(1,1)$--Hessian of function $f$.
\end{lem}
\begin{proof}
By using Lemma \ref{lem:NablaFEucl}, we obtain
\begin{equation*}
\begin{split}\nabla_X \nabla f &= \top \left( \wnabla_X \nabla f \right) = \top \left( \wnabla_X(p - h\nu)\right)\\
&= \top (X - X(h)\nu - hB(X)) = X - hB(X),
\end{split}
\end{equation*}
which proves the lemma.
\end{proof}

Let $\xi$ be an infinitesimal deformation of $M$, and let $M_t = \phi_t(M)$ be the geodesic flow of $M$ along $\xi$, see Subsection \ref{subsec:BLTheorem}. Put $f_t \colon M_t \to \R$, $f_t(x) = \|x\|^2/2$ and
$$
\dot f = \lim_{t \to 0} \frac{f_t \circ \phi_t - f}{t}.
$$

\begin{lem}
\label{lem:HessDotF}
If $\xi$ is an isometric infinitesimal deformation, then we have
$$
\Hess \dot f = - (\dot h B + h \dot B),
$$
where $\dot B$ is defined as in Subsection \ref{subsec:VarShape}.
\end{lem}
\begin{proof}
The isometricity of $\xi$ implies
$$
\Hess \dot f = (\Hess f)^\bigdot,
$$
where $(\Hess f)^\bigdot$ is defined similarly to $\dot B$. Now the lemma follows by differentiating equation \eqref{eqn:HessFEucl}.
\end{proof}

Consider the differential 1--form
$$
\darea_M(\nabla \dot f \wedge \dot B) \in \Omega^1(M).
$$
Here $\nabla \dot f$ is viewed as a $TM$--valued 0--form, and $\dot B$ as a $TM$--valued 1--form on $M$. By Subsection \ref{subsec:BundleVal}, we have
\begin{equation}
\label{eqn:StokesLS}
d(\darea_M(\nabla \dot f \wedge \dot B)) = \darea_M(\Hess \dot f \wedge \dot B) + \darea_M(\nabla \dot f \wedge d^\nabla \dot B),
\end{equation}
where $d^\nabla \colon \Omega^1(M, TM) \to \Omega^2(M, TM)$ is the exterior differential associated with Levi-Civita connection $\nabla$ on $M$. Isometricity of deformation $\xi$ and the Codazzi-Mainardi equations imply
$$
d^\nabla \dot B = (d^\nabla B)^\bigdot = 0,
$$
so that the second summand on the right hand side of \eqref{eqn:StokesLS} vanishes. Further, by Lemma \ref{lem:HessDotF} and because of $2 \det(B, \dot B) = (\det B)^\bigdot = \dot K = 0$ we have
$$
\darea_M(\Hess \dot f \wedge \dot B) = - 2 (\dot h \det(B, \dot B) + h \det \dot B) = -2 h \det \dot B.
$$
By integrating \eqref{eqn:StokesLS} and applying Stokes' theorem we arrive to \eqref{eqn:IntHDotB}.

\begin{rem}
A similar argument is used in the proof of \cite[Lemma 3.1]{LS00} that deals with infinitesimal rigidity in the Minkowski space.
\end{rem}

The following lemma allows to relate the above argument to the argument in Subsection \ref{subsec:OtherProof}.

\begin{lem}
Let $\tau$ be the translation vector field of an isometric infinitesimal deformation $\xi$. Then we have
$$
\nabla \dot f = \top \tau.
$$
\end{lem}
\begin{proof}
By differentiating the equation $f_t = \|p_t\|^2/2$, we obtain $\dot f = \langle p, \xi \rangle$. Hence for all $X \in TM$ we have
$$
\langle X, \nabla \dot f \rangle = \nabla_X \dot f = \langle X, \xi \rangle + \langle p, \eta \times X \rangle = \langle X, \xi + p \times \eta \rangle = \langle X, \tau \rangle,
$$
where we used equations from Lemma \ref{lem:EtaTau}. This implies the statement of the lemma.
\end{proof}

Together with equation \eqref{eqn:DTauBDot} this implies that the differential 1--forms used here and in Subsection \ref{subsec:OtherProof} are closely related:
$$
\darea_M(\nabla \dot f \wedge \dot B) = \dvol(\nu \wedge \top \tau \wedge (h^{-1} \top \circ d\tau)) = \dvol(h^{-1}\nu \wedge \tau \wedge d\tau).
$$

\section{Metric rigidity and the Hilbert-Einstein functional}
\label{sec:RigHE}
In this section we provide a variational interpretation of the proof of Theorem \ref{thm:InfRigSm} given in Subsections \ref{subsec:OtherProof} and \ref{subsec:HessSquaredDist}. This interpretation does not simplify the arguments, rather conversely. But it supports the ideas suggested in Section \ref{sec:Directions} concerning approaches to the Weyl problem and to the pleating lamination conjecture.

\subsection{Reduction to warped product deformations}
\label{subsec:RedWarp}
Let $M \subset \R^3$ be a smooth convex closed surface such that the coordinate origin $0$ lies in the interior of the body $P$ bounded by $M$. Instead of deforming the embedding of $M$ into $\R^3$, we will deform the metric inside $P$ in the following way. Think of $P$ as composed of thin pyramids with apex at $0$ and bases on $M$; vary the lengths of lateral edges of the pyramids while leaving the base edge lengths constant and thus preserving the intrinsic metric of $M$.

If the lengths of lateral edges are given by a smooth function $r \colon M \to \R$, then this construction gives a smooth metric $\wg_r$ on $P \setminus \{0\}$, called a \emph{warped product metric}. For some functions $r$, the metric $\wg_r$ is Euclidean, for example if $r$ is the distance from some interior point $a \in P$ (if we are just moving the origin inside $P$, so to say). The key point is that the surface $M$ is infinitesimally rigid if and only if no variations of $r$ besides those mentioned above leave the metric $\wg_r$ Euclidean in the first order. The goal of this subsection is to explain this equivalence and to state a reformulation of Theorem~\ref{thm:InfRigSm}.

Consider the radial projection
\begin{gather*}
M \to \Sph^2,\\
y \mapsto \frac{y}{\|y\|},
\end{gather*}
and denote by $\phi_0$ its inverse.

\begin{dfn}
\label{dfn:DistFunc}
The \emph{distance function} of a surface $M$ is
\begin{gather*}
r_0 \colon \Sph^2 \to \R_+,\\
r_0(x) = \|\phi_0(x)\|.
\end{gather*}
\end{dfn}

It follows that the surface $M$ is the image of an embedding
\begin{gather}
\phi_0 \colon \Sph^2 \to \R^3, \nonumber\\
\phi_0(x) = r_0(x) \cdot x. \label{eqn:phi=rx}
\end{gather}

\begin{lem}
\label{lem:IndMetr}
Consider the diffeomorphism
\begin{gather*}
F \colon \R_+ \times \Sph^2 \to \R^3 \setminus \{0\},\\
(\rho, x) \mapsto \rho x.
\end{gather*}
Then the pullback of the canonical Euclidean metric on $\R^3 \setminus \{0\}$ to $\R_+ \times \Sph^2$ by the map $F$ is given by
$$
F^*(\can) = d\rho^2 + \left( \frac{\rho}{r_0} \right)^2 (g - dr_0 \otimes dr_0),
$$
where $g = \phi_0^*(\can)$ is the metric induced on $\Sph^2$ by the embedding \eqref{eqn:phi=rx}.
\end{lem}
\begin{proof}
We know that $F^*(\can) = d\rho^2 + \rho^2 \can_{\Sph^2}$. Thus we have to show that
$$
\can_{\Sph^2} = \frac{1}{r_0^2} (g - dr_0 \otimes dr_0),
$$
or, equivalently, $g = r_0^2 \can_{\Sph^2} + dr_0 \otimes dr_0$. The last equation follows easily by taking the differential of \eqref{eqn:phi=rx}.
\end{proof}

Lemma \ref{lem:IndMetr} motivates the following construction. Let $g$ be a Riemannian metric on $\Sph^2$, and let $r \colon \Sph^2 \to \R_+$ be a smooth function such that
\begin{equation}
\label{eqn:GradR}
\|\nabla r\|_g < 1
\end{equation}
everywhere on $\Sph^2$. Associate with $(g,r)$ a Riemannian metric
\begin{equation}
\label{eqn:TildeGR}
\widetilde{g}_r = d\rho^2 + \left(\frac{\rho}{r}\right)^2 (g - dr \otimes dr)
\end{equation}
on $\R_+ \times \Sph^2$. (Condition \eqref{eqn:GradR} ensures that the symmetric bilinear form $g - dr \otimes dr$ is positive definite.) It is easy to see that the map
\begin{gather}
\phi_r \colon \Sph^2 \to \R_+ \times \Sph^2, \nonumber\\
\phi_r(x) = (r(x), x) \label{eqn:PhiR}
\end{gather}
embeds $(\Sph^2, g)$ isometrically into $(\R_+ \times \Sph^2, \wg_r)$.

Assume that the metric \eqref{eqn:TildeGR} is flat, i.~e. has a vanishing curvature tensor. Then (see \cite{SpiII}, \cite[Chapter 2.3]{Wolf11}) the Riemannian manifold $(\R_+ \times \Sph^2, \wg_r)$ is locally isometric to $\R^3$, moreover there is an isometry $\R_+ \times \Sph^2 \to \R^3 \setminus \{0\}$, since the source space is simply-connected and the metric $\wg_r$ is complete. It follows that the map \eqref{eqn:PhiR} composed with this isometry is an isometric embedding of $(\Sph^2, g)$ into~$\R^3$.

In particular, if the metric $g$ is induced by an embedding $\phi_0 \colon \Sph^2 \to \R^3$ and $r_0(x) = \|\phi_0(x)\|$, then the metric $\wg_{r_0}$ is flat. By the previous paragraph, embedding $\phi_0$ is determined by $r_0$ up to an isometry of $\R^3$ fixing $0$.

\begin{rem}
The embedding $\phi_0$ from the previous paragraph need not be the inverse of the radial projection. Indeed, if the metric $\wg_r$ is flat, then so is the metric determined by $g' = \psi^* g$ and $r' = r \circ \psi$ for any diffeomorphism $\psi \colon \Sph^2 \to \Sph^2$. The distinctive feature of the parametrization \eqref{eqn:phi=rx} is that the corresponding isometry $(\R_+ \times \Sph^2, \wg_r) \to (\R^3 \setminus \{0\}, \can)$ is the map $\Phi$ in Lemma \ref{lem:IndMetr}.
\end{rem}

Consider a smooth function $s \colon \Sph^2 \to \R$ and put
$$
r_t = r_0 + ts.
$$
As $r_0$ satisfies condition \eqref{eqn:GradR}, $r_t$ also satisfies \eqref{eqn:GradR} for all sufficiently small~$t$. Let $\wR_t$ be the curvature tensor of the metric $\wg_{r_t}$ on $\R_+ \times \Sph^2$; denote
$$
\dot R = \lim_{t \to 0} \frac{\wR_t - \wR_0}{t}.
$$

\begin{dfn}
\label{dfn:SCurvPres}
A function $s \colon \Sph^2 \to \R$ is called a \emph{curvature preserving infinitesimal deformation} of $r_0$, if $\dot R = 0$.
\end{dfn}

\begin{dfn}
\label{dfn:STriv}
A function $s \colon \Sph^2 \to \R$ is called a \emph{trivial infinitesimal deformation} of $r_0$, if there exists an $a \in \R^3$ such that
\begin{equation}
\label{eqn:STriv}
s(x) = \left\langle a, \frac{\phi_0(x)}{\|\phi_0(x)\|} \right\rangle.
\end{equation}
\end{dfn}

As the embedding $\phi_0$ is determined by $r_0$ up to an isometry of $\R^3$ fixing $0$, the class of trivial deformations is well-defined.

\begin{lem}
\label{lem:STrivCurvPres}
If $s$ is trivial in the sense of Defitinion \ref{dfn:STriv}, then $s$ is curvature preserving in the sense of Definition \ref{dfn:SCurvPres}.
\end{lem}
\begin{proof}
Let $s$ be as in \eqref{eqn:STriv}. Consider the function
\begin{gather*}
r'_t(x) = \|\phi_0(x) + ta\|.
\end{gather*}
As $r'_t$ is the distance function of the surface $M$ translated by the vector $-ta$, the metric $\wg_{r'_t}$ is flat. But since $r'_t - r_t = o(t)$ in the $C^2$--norm, the metric $\wg_{r_t}$ is flat in the first order. Thus $\dot R = 0$, and $s$ is curvature preserving.
\end{proof}

Now we can state the main result of this subsection.
\begin{thm}
\label{thm:RedToR}
A smooth convex closed surface $M$ is infinitesimally rigid if and only if every curvature preserving deformation of its distance function is trivial.
\end{thm}

\begin{proof}
The theorem is proved by establishing a correspondence between isometric deformations $\xi$ of $M$ and curvature preserving deformations $s$ of the distance function $r_0$ so that trivial deformations correspond to trivial ones.

Let $\xi \colon \Sph^2 \to \R^3$ be a vector field along an embedding $\phi_0 \colon \Sph^2 \to \R^3$. Define the function $s \colon \Sph^2 \to \R$ as
\begin{equation}
\label{eqn:XiToS}
s(x) = \left\langle \xi(x), \frac{\phi_0(x)}{\|\phi_0(x)\|} \right\rangle.
\end{equation}
We claim that if $\xi$ is an isometric infinitesimal deformation, then $s$ is curvature preserving. The proof is similar to that of Lemma \ref{lem:STrivCurvPres}. Consider the embedding $\phi_t = \phi_0 + t\xi$. Let $r'_t = \|\phi_t\|$ be its distance function, and let $g^t$ be the metric induced on $\Sph^2$ by $\phi_t$. It follows that the metric on $\R_+ \times \Sph^2$ obtained by substituting in \eqref{eqn:TildeGR} $g^t$ for $g$ and $r'_t$ for $r$ is flat. As $r'_t - r_t = o(t)$, and $g_t - g = o(t)$ in the $C^2$--norm, the metric $\wg_{r_t}$ is flat in the first order.

Now let $s$ be such that $\dot R = 0$. By \cite[Proposition 3]{Cal61}, an ifninitesimal deformation $\widetilde{h}$ of a flat Riemannian metric $\wg$ leaving the curvature zero in the first order is locally induced by an infinitesimal diffeomorphism; i.~e. $\widetilde{h}$ is the Lie derivative of $\wg$ along a vector field $\eta$. As $\R_+ \times \Sph^2$ is simply-connected, the local vector field $\eta$ can be extended to a global one, so that we have
$$
\left. \frac{d}{dt} \right|_{t=0} \wg_{r_t} = {\mathcal L}_\eta\, \wg_{r_0}.
$$
Note that $\eta$ is defined uniquely up to adding a Killing vector field with respect to the flat metric $\wg_{r_0}$. It can be shown that the limit
$$
\eta(0) := \lim_{\rho \to 0} \eta(\rho, x)
$$
exists. Add to $\eta$ a parallel translation so that $\eta(0) = 0$ and put
$$
\xi(x) = \eta(r_0(x), x).
$$
It is easy to see that $\xi$ is an isometric infinitesimal deformation of the surface
$$
M = \{(r_0(x), x)\} \subset \R_+ \times \Sph^2.
$$
Besides, the variation $s$ of the distance of the point $(r_0(x), x)$ from the singular point $\rho = 0$ is related to $\xi$ through the formula \eqref{eqn:XiToS}. (This is due to the fact that $\eta(0) = 0$ and that the lines $\R_+ \times \{x\}$ are geodesics in the metric $\wg_r$.) It follows that the map $s \mapsto \xi$ thus obtained is the inverse of the map constructed in the previous paragraph. Note that $\xi$ is well-defined up to adding a rotation around $0$.

Trivial deformations are sent to trivial ones; in particular, if $s$ satisfies \eqref{eqn:STriv}, then we can put
$$
\eta(\rho,x) = \frac{\rho}{r_0(x)} a
$$
which results in $\xi = a$.
\end{proof}

The curvature tensor of the metric $\wg_r$ is completely determined by sectional curvatures in some small set of planes, see Section \ref{sec:CurvWarped}. For example, it suffices to take all planes tangent to the surface
\begin{equation}
\label{eqn:MR}
M_r = \{(r(x), x)\}.
\end{equation}
\begin{dfn}
Denote by $\sec(r)(x)$ the sectional curvature of the Riemannian manifold $(\R_+ \times \Sph^2, \wg_r)$ in the plane $T_{(r(x), x)} M_r$. We will usually omit the argument $r$ of $\sec$.
\end{dfn}

Let $s = \dot r$ be an arbitrary infinitesimal deformation of the function $r$. Denote by $\sec^\bigdot$ the derivative of $\sec$ in the direction $\dot r$:
$$
\sec^\bigdot(x) = \lim_{t \to 0} \frac{\sec(r + t \dot r)(x) - \sec(r)(x)}{t}.
$$

\begin{cor}
\label{cor:RedToSec}
A smooth convex closed surface $M = \phi_0(\Sph^2)$ is infinitesimally rigid if and only if every infinitesimal deformation $\dot r$ of its distance function $r_0$ such that $\sec^\bigdot = 0$ is trivial.
\end{cor}

\subsection{Hessian of the squared distance function in a warped product}
Consider the function
\begin{gather*}
f \colon \R_+ \times \Sph^2 \to \R,\\
f(\rho, x) = \frac{\rho^2}2,
\end{gather*}
and denote by ${\widetilde\Hess} f$ its Hessian with respect to the metric $\wg_r$ in \eqref{eqn:TildeGR}, ${\widetilde\Hess} f(X) = \wnabla_X \wnabla f$.

\begin{lem}
\label{lem:HessTilde}
$$
{\widetilde\Hess} f = \id
$$
\end{lem}
\begin{proof}
We have (see \cite[Chapter 2, Section 1.3]{Pet06})
$$
\widetilde g_r \left( {\widetilde\Hess} f(X), Y \right) = \frac12 {\mathcal L}_{\wnabla f}\, \widetilde g_r (X, Y).
$$
The flow $F_t$ on $\R_+ \times \Sph^2$ generated by $\wnabla f$ scales the $\rho$-coordinate by $e^t$. It follows that
$$
{\mathcal L}_{\wnabla f}\, \widetilde g_r = 2\widetilde g_r.
$$
Thus ${\widetilde\Hess} f(X) = X$ for all $X$, and the lemma is proved.
\end{proof}

We now introduce several vector fields along the surface $M_r \subset \R_+ \times \Sph^2$ given by \eqref{eqn:MR}. First, there is the field of outward unit normals $\nu$. Second, let $\partial_\rho$ be the radial unit vector field on $\R_+ \times \Sph^2$. The vector field $\rho \partial_\rho$ generalizes the position vector field $p$ that we considered in the case of a surface $M \subset \R^3$. Since $\partial_\rho$ has a unit norm and is orthogonal to $T\Sph^2$ with respect to $\widetilde g_r$, it is also the gradient field of the function $\rho \colon (\rho, x) \mapsto \rho$. It follows that
\begin{equation}
\label{eqn:WnablaF}
\wnabla f = \rho \wnabla \rho = \rho \partial_\rho.
\end{equation}

Denote the restriction of the function $\rho$ to $M_r$ by $r$:
\begin{gather*}
r \colon M_r \to \R,\\
r = \rho|_{M_r}.
\end{gather*}
(This is an abuse of notation, as earlier we denoted by $r$ the function $\phi_r^{-1} \circ \rho$ on $\Sph^2$.) Let us introduce two auxiliary functions on $M_r$.

\begin{dfn}
\label{dfn:AlphaH}
Denote by $\alpha \colon M_r \to [0, \frac\pi 2)$ the angle between vectors $\partial_\rho$ and $\nu$, and define $h \colon M_r \to \R_+$ as $h = r \cos\alpha$.
\end{dfn}

The function $h$ generalizes the support function defined in the case of a surface $M \subset \R^3$.

\begin{figure}[ht]
\begin{center}
\input{Gradients.tex}
\end{center}
\caption{The gradients of $f=\frac{\rho^2}2$.}
\label{fig:Gradients}
\end{figure}
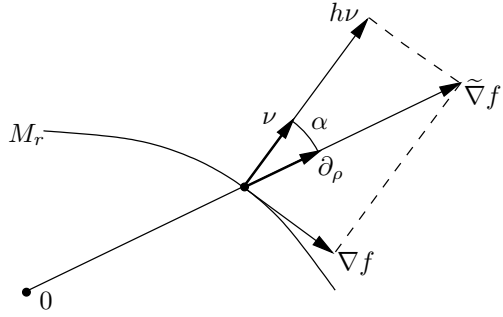

\begin{lem}
\label{lem:HessFShape}
Let $\nabla$ be the covariant derivative on $M_r$ associated with the metric $g = \wg_r|_{M_r}$. Let $\Hess f: X \mapsto \nabla_X \nabla f$ be the Hessian of $f|_{M_r}$. Then we have
\label{lem:Hess}
\begin{equation}
\label{eqn:Hess}
\Hess f = \id - hB,
\end{equation}
where $B \colon X \mapsto \wnabla_X\nu$ is the shape operator on $M_r$.
\end{lem}
\begin{proof}
By \eqref{eqn:WnablaF} and Definition \ref{dfn:AlphaH} we have $\nabla f = \top(\wnabla f) = \wnabla f - h\nu$. Further,
\begin{equation*}
\begin{split}
\nabla_X \nabla f &= \top(\wnabla_X \nabla f) = \top(\wnabla_X(\wnabla f - h\nu))\\
&= \top({\widetilde\Hess} f(X) - X(h)\nu - h B(X)) = X - h B(X),
\end{split}
\end{equation*}
which proves the lemma.
\end{proof}

\subsection{The Hilbert-Einstein functional}
Denote
$$
P_r = \{(\rho, x) \in \R_+ \times \Sph^2\ |\ \rho \le r(x)\}.
$$

\begin{dfn}
The \emph{Hilbert-Einstein functional} of the Riemannian manifold with boundary $(P_r, \wg_r)$ is defined as
\begin{equation}
\label{eqn:HEDef}
\HE(r) = \frac12 \int_{P_r} \scal \, \dvol + 2 \int_{M_r} H\, \darea,
\end{equation}
where $\scal$ is the scalar curvature of the metric $\wg_r$, and $H = \frac12 \tr B$ is the total mean curvature of the surface $M_r = \partial P_r$.
\end{dfn}

The manifold $D_r$ is not compact (as we exclude the point $\rho = 0$). However the first integral in \eqref{eqn:HEDef} converges; this follows from the next lemma that expresses it as an integral over $M_r$.

Recall that $\sec(x)$ denotes the sectional curvature of $\wg_r$ in the plane tangent to $M_r$ at $(r(x), x)$.

\begin{lem}
\label{lem:TotScal}
$$
\int_{P_r} \scal \, \dvol = 2 \int_{M_r} r \frac{\sec}{\cos\alpha}  \, \darea,
$$
\end{lem}
\begin{proof}
Put $\hat g = \frac{g - dr \otimes dr}{r^2}$, so that we have
$$
\wg_r = d\rho^2 + \rho^2 \hat g.
$$
Then $\dvol = \rho^2 \cdot d\rho \wedge \darea_{\hat g}$ and
$$
\int_{P_r} \scal \, \dvol = \int_{\Sph^2} \darea_{\hat g} \int_0^{r(x)} \rho^2 \scal \, d\rho.
$$
Equations \eqref{eqn:SecL} and \eqref{eqn:Sec1} imply
$$
\scal(\rho, x) = 2\sec_{(\rho,x)}(\partial_\rho^\perp) = 2 \frac{\sec_{(1,x)}(\partial_\rho^\perp)}{\rho^2} .
$$
Substituting this in the previous equation yields
\begin{equation}
\label{eqn:IntScalFirst}
\int_{P_r} \scal\, \dvol = 2 \int_{\Sph^2} r \cdot \sec_{(1,x)}(\partial_\rho^\perp) \, \darea_{\hat g}.
\end{equation}

We now want to rewrite this as an integral over $M_r$. First, note that the Jacobian of the radial projection $\pi \colon \{1\} \times \Sph^2 \to M_r$ equals $\frac{r^2}{\cos\alpha}$, so that
$$
\darea_{\hat g} = \frac{\cos\alpha}{r^2} \cdot \pi^*(\darea_g).
$$
And secondly, again by \eqref{eqn:Sec1} and \eqref{eqn:SecL}, we have
$$
\sec_{(1, x)}(\partial_\rho^\perp) = r(x)^2 \cdot \sec_{(r(x), x)}(\partial_\rho^\perp) = r^2\frac{\sec}{\cos^2 \alpha}.
$$
By substituting both equations in \eqref{eqn:IntScalFirst}, we obtain
$$
\int_{P_r} \scal \, \dvol = 2 \int_{M_r} r \cdot r^2 \frac{\sec}{\cos^2 \alpha} \cdot \frac{\cos\alpha}{r^2} \, \darea = 2 \int_{M_r} r \frac{\sec}{\cos\alpha} \, \darea,
$$
and the lemma is proved.
\end{proof}

\begin{thm}
\label{thm:HE}
$$
\HE(r) = \int_{M_r} h(K + \det B)\, \darea,
$$
where $K$ is the Gauss curvature of $M_r$, and $h$ is as in Definition \ref{dfn:AlphaH}.
\end{thm}
\begin{proof}
Consider the differential 2--form
$$
\darea(\Hess f \wedge B)
$$
on $M_r$. Here the linear operators $\Hess f$ and $B$ are viewed as $TM_r$--valued 1--forms; their wedge product and operation of $\darea$ on it are defined in Subsection \ref{subsec:MixDetVVal}. We have
$$
\Hess f = \nabla (\nabla f) = d^\nabla (\nabla f),
$$
where $d^\nabla$ is the exterior derivative on $TM_r$--valued forms. Thus by Lemmas \ref{lem:DDet} and \ref{lem:LeibVect} we have
$$
d(\darea(\nabla f \wedge B)) = \darea(\Hess f \wedge B) + \darea(\nabla f \wedge d^\nabla B).
$$
Stokes' theorem implies
\begin{equation}
\label{eqn:PartInt0}
\int_{M_r} \darea(\Hess f \wedge B) + \int_{M_r} \darea(\nabla f \wedge d^\nabla B) = 0.
\end{equation}

With the help of Lemmas \ref{lem:DvolDet} and \ref{lem:HessFShape} the first integrand in \eqref{eqn:PartInt0} can be computed as
\begin{equation}
\label{eqn:AHessB}
\darea(\Hess f \wedge B) = 2(H - h \det B) \darea. 
\end{equation}

To compute the second integrand in \eqref{eqn:PartInt0}, observe that
\begin{eqnarray*}
d^{\nabla}B(X,Y) & = & \nabla_X (B(Y)) - \nabla_Y (B(X)) - B([X,Y])\\
& = & \nabla_X \wnabla_Y \nu - \nabla_Y \wnabla_X \nu - \wnabla_{[X,Y]}\nu\\
& = & \top(\wnabla_X \wnabla_Y \nu - \wnabla_Y \wnabla_X \nu - \wnabla_{[X,Y]}\nu)\\
& = & \top(\wR(X, Y)\nu).
\end{eqnarray*}
By \eqref{eqn:R} we have
$$
\wR(X, Y)\nu = \sec(\partial_\rho^\perp) \cdot \dvol(\partial_\rho, X, Y) \cdot (\nu \times \partial_\rho) \in TM_r.
$$
Note that
$$
\dvol(\partial_\rho, X, Y) = \cos\alpha \cdot \darea(X,Y), \quad \nu \times \partial_\rho = J(\nabla r),
$$
where $J \colon TM_r \to TM_r$ is a rotation by $\frac\pi2$. Besides, by \eqref{eqn:SecL} we have
$$
\sec(\partial_\rho^\perp) = \frac{\sec}{\cos^2 \alpha}.
$$
As a result, we have
\begin{equation}
\label{eqn:DNB}
d^\nabla B = \frac{\sec}{\cos\alpha} J(\nabla r) \darea,
\end{equation}
so that
\begin{equation}
\label{eqn:AfDB}
\darea(\nabla f \wedge d^\nabla B) = \langle \nabla f, \nabla r \rangle \frac{\sec}{\cos\alpha} \darea.
\end{equation}

By substituting \eqref{eqn:AHessB} and \eqref{eqn:AfDB} in \eqref{eqn:PartInt0} and using the identity
$$
\langle \nabla f, \nabla r \rangle = r \|\nabla r\|^2 = r\sin^2\alpha,
$$
we obtain
$$
\int_{M_r} H\, \darea = \int_{M_r} h \det B\, \darea - \frac12 \int_{M_r} r \sin^2\alpha\, \frac{\sec}{\cos\alpha} \, \darea.
$$
Together with Lemma \ref{lem:TotScal} this implies
\begin{eqnarray*}
\HE(r) & = & \int_{M_r} \left( r \frac{\sec}{\cos\alpha} + 2h \det B - r \sin^2\alpha\, \frac{\sec}{\cos\alpha} \right)\, \darea\\
& = & \int_{M_r} (h \sec + 2h \det B)\, \darea = \int_{M_r} h(K + \det B)\, \darea.
\end{eqnarray*}
Here we used the identity $h = r \cos\alpha$ and the Gauss equation
\begin{equation}
\label{eqn:Gauss}
K = \sec + \det B.
\end{equation}
The theorem is proved.
\end{proof}

\begin{rem}
For a surface $M$ in $\R^3$ the same argument yields one of the Minkowski formulas:
$$
\int_M H \, \darea = \int_M hK \, \darea,
$$
see \cite[Chapter 12]{SpiV}. The other Minkowski formula
$$
\Area(M) = \int_M hH \, \darea
$$
also holds in our more general situation and can be proved by integrating the 2--form $\darea(\Hess f \wedge \id)$.
\end{rem}

\subsection{First derivative of the Hilbert-Einstein functional}
The Hil\-bert-Einstein functional is a differentiable map from $C^\infty(\Sph^2)$ to $\R$. Denote by $\HE^\bigdot$ the derivative of $\HE$ in the direction of $\dot r \in C^\infty(\Sph^2)$.

\begin{thm}
\label{thm:HEDot}
\begin{equation}
\label{eqn:HEDot}
\HE^\bigdot = \int_{M_r} \dot r \frac{\sec}{\cos \alpha}\, \darea
\end{equation}
\end{thm}
\begin{proof}
Similar to the proof of Theorem \ref{thm:HE}, we integrate by parts the differential 2--form
$$
\darea(\Hess \dot f \wedge B).
$$
We have
\begin{equation}
\label{eqn:PartInt1}
\int_{M_r} \darea(\Hess \dot f \wedge B) + \int_{M_r} \darea(\nabla \dot f \wedge d^\nabla B) = 0.
\end{equation}

As the metric $g$ on $M_r$ does not depend on $r$ (rather, one should speak about the metric induced on $\Sph^2$ by $\phi_r \colon \Sph^2 \to M_r$), we have $\dot \nabla = 0$, and hence
$$
\Hess \dot f = (\Hess f)^\bigdot = - \dot h B - h \dot B.
$$
Thus the first integrand in \eqref{eqn:PartInt1} equals
$$
\darea(\Hess \dot f \wedge B) = -2(\dot h \det B + h \det(\dot B, B)) \darea.
$$
By Lemma \ref{lem:DerDet}, we have $2 \det(\dot B, B) = (\det B)^\bigdot$. Besides, due to the constancy of the metric $g$ we have $\dot K = 0$, so that the Gauss equation \eqref{eqn:Gauss} implies $(\det B)^\bigdot = - \sec^\bigdot$. As a result, we obtain
\begin{equation}
\label{eqn:ADotHessB}
\darea(\Hess \dot f \wedge B) = (-2 \dot h \det B + h \sec^\bigdot) \darea.
\end{equation}

To compute the second integrand in \eqref{eqn:PartInt1}, substitute \eqref{eqn:DNB}:
$$
\darea(\nabla \dot f \wedge d^\nabla B) = \langle \nabla \dot f, \nabla r \rangle \frac{\sec}{\cos\alpha} \darea.
$$
An easy computation yields
\begin{equation}
\label{eqn:DotFR}
\langle \nabla \dot f, \nabla r \rangle = \dot r - \dot h \cos\alpha,
\end{equation}
so that we have
\begin{equation}
\label{eqn:ADotFDB}
\darea(\nabla \dot f \wedge d^\nabla B) = \left( \dot r \frac{\sec}{\cos\alpha} - \dot h \sec \right) \darea.
\end{equation}

By substituting \eqref{eqn:ADotHessB} and \eqref{eqn:ADotFDB} in \eqref{eqn:PartInt1}, we obtain
$$
\int_{M_r} \left( -2 \dot h \det B + h \sec^\bigdot + \dot r \frac{\sec}{\cos\alpha} - \dot h \sec \right) \darea = 0.
$$
On the other hand, by differentiating the formula from Theorem \ref{thm:HE} we get
\begin{eqnarray*}
\HE^\bigdot & = & \int_{M_r} \left( \dot h(K + \det B) + h (\det B)^\bigdot \right) \darea\\
& = & \int_{M_r} \left( \dot h \sec + 2 \dot h \det B - h \sec^\bigdot \right) \darea
\end{eqnarray*}
By substituting this in the previous formula, we obtain \eqref{eqn:HEDot}. The theorem is proved.
\end{proof}

\begin{rem}
By differentiating the formula in Lemma \ref{lem:TotScal} and combining the result with \eqref{eqn:HEDot}, we obtain
$$
\int_{M_r} \dot H \darea = - \frac12 \int_{M_r} r \left( \frac{\sec}{\cos\alpha} \right)^\bigdot \darea.
$$
The same can be proved by integrating $\darea(\Hess f \wedge \dot B)$. It follows that the derivative of the total mean curvature vanishes if $\left(\frac{\sec}{\cos\alpha}\right)^\bigdot$ vanishes at every point. This happens in particular for $\sec = 0 = \sec^\bigdot$, that is for isometric infinitesimal deformations of a surface in $\R^3$. See \cite{AR98,RS99,Sou99,AleV10}.
\end{rem}

\begin{rem}
Theorem \ref{thm:HEDot} can also be derived from the general formula for the derivative of the Hilbert-Einstein functional on a manifold with boundary, see e.~g. \cite[Equation (2.9)]{Ara03}, provided that care is taken of the singularity at $\rho=0$. For this, one has to estimate the decay of the derivative of scalar curvature at $\rho \to 0$, which is easily done.
\end{rem}

\subsection{Second derivative of $\HE$ and the proof of Theorem \ref{thm:InfRigSm}}
\label{subsec:HEDDot}
Denote by $\HE^{\bigdot\bigdot}$ the second derivative of $\HE$ in the direction $\dot r$:
$$
\HE^{\bigdot\bigdot} = \left. \frac{d^2}{dt^2} \right|_{t=0} \HE(r + t \dot r) = (\HE^\bigdot)^\bigdot.
$$
The second derivative is a quadratic form in $\dot r$.

\begin{thm}
\label{thm:HEDDot}
$$
\HE^{\bigdot\bigdot} = \int_{M_r} \dot r \left( \frac{\sec}{\cos\alpha} \right)^\bigdot \darea,
$$
$$
\HE^{\bigdot\bigdot} = \int_{M_r} 2h \det \dot B \, \darea + \int_{M_r} r (\cos\alpha)^{\bigdot\bigdot} \sec \, \darea.
$$
\end{thm}
\begin{proof}
The first formula of the theorem follows by differentiating \eqref{eqn:HEDot}. The second formula is proved fully in the spirit of Theorems \ref{thm:HE} and \ref{thm:HEDot}, by integrating the differential 2--form
$$
\darea(\Hess \dot f \wedge \dot B).
$$
We have
\begin{equation}
\label{eqn:PartInt2}
\int_{M_r} \darea(\Hess \dot f \wedge \dot B) + \int_{M_r} \darea(\nabla \dot f \wedge d^\nabla \dot B) = 0.
\end{equation}
Compute the first integrand:
\begin{equation}\label{eqn:ADotHessDotB}
\darea(\Hess \dot f \wedge \dot B) = (\dot h \sec^\bigdot - 2h \det \dot B) \darea. 
\end{equation}
For the second integrand, observe that \eqref{eqn:DNB} implies
$$
d^\nabla \dot B = J \left( \frac{\sec}{\cos\alpha} \nabla r \right)^\bigdot \darea,
$$
due to the constancy of $g$ and of the associated operators $\nabla$ and $d^\nabla$. Thus we have
\begin{equation}
\label{eqn:ANFDB}
\begin{split}
\darea(\nabla \dot f \wedge d^\nabla \dot B) & = \left\langle \nabla \dot f, \left( \frac{\sec}{\cos\alpha} \nabla r \right)^\bigdot \right\rangle \darea \\
&= \left( \langle \nabla \dot f, \nabla \dot r \rangle \frac{\sec}{\cos\alpha} + \langle \nabla \dot f, \nabla r \rangle \left( \frac{\sec}{\cos\alpha} \right)^\bigdot \right) \darea.
\end{split}
\end{equation}
One easily computes $\langle \nabla f, \nabla \dot r \rangle = -h (\cos\alpha)^\bigdot$, hence
$$
\langle \nabla \dot f, \nabla \dot r \rangle = \langle \nabla f, \nabla \dot r \rangle^\bigdot = -\dot h (\cos\alpha)^\bigdot - h (\cos\alpha)^{\bigdot\bigdot}.
$$
By substituting this and \eqref{eqn:DotFR} in \eqref{eqn:ANFDB}, we compute the scalar factor at $\darea$:
\begin{multline*}
\left( -\dot h (\cos\alpha)^\bigdot - h (\cos\alpha)^{\bigdot\bigdot} \right) \frac{\sec}{\cos\alpha} + \left( \dot r - \dot h \cos\alpha \right) \left( \frac{\sec}{\cos\alpha} \right)^\bigdot\\
= - \dot h \left( (\cos\alpha)^\bigdot \frac{\sec}{\cos\alpha} + \cos\alpha \left( \frac{\sec}{\cos\alpha} \right)^\bigdot \right) - r (\cos\alpha)^{\bigdot\bigdot} \sec + \dot r \left( \frac{\sec}{\cos\alpha} \right)^\bigdot\\
= - \dot h \sec^\bigdot - r (\cos\alpha)^{\bigdot\bigdot} \sec + \dot r \left( \frac{\sec}{\cos\alpha} \right)^\bigdot.
\end{multline*}
We finally obtain
\begin{equation}
\label{eqn:ANFDB1}
\darea(\nabla \dot f \wedge d^\nabla \dot B) = \left( - \dot h \sec^\bigdot - r (\cos\alpha)^{\bigdot\bigdot} \sec + \dot r \left( \frac{\sec}{\cos\alpha} \right)^\bigdot \right) \darea.
\end{equation}

By substituting \eqref{eqn:ADotHessDotB} and \eqref{eqn:ANFDB1} in \eqref{eqn:PartInt2}, we obtain
$$
\int_{M_r} \left( -2h \det \dot B - r (\cos\alpha)^{\bigdot\bigdot} \sec + \dot r \left( \frac{\sec}{\cos\alpha} \right)^\bigdot \right) \darea = 0
$$
Being combined with the first equation of the theorem, this implies the second equation. The theorem is proved.
\end{proof}

\begin{proof}[Proof of Theorem \ref{thm:InfRigSm}]
Let $r = r_0$ be the distance function of an embedding $\phi_0 \colon \Sph^2 \to \R^3$ and let $\dot r$ be the variation of $r$ that corresponds to an infinitesimal isometric deformation $\xi$. By Theorem \ref{thm:RedToR}, we have $\sec^\bigdot = 0$.

As the metric $\wg_{r_0}$ is flat we have $\sec = 0$. It follows that
$$
\left( \frac{\sec}{\cos\alpha} \right)^\bigdot = \frac{\sec^\bigdot}{\cos\alpha} + \sec \left( \frac{1}{\cos\alpha} \right)^\bigdot = 0.
$$
Thus the first formula from Theorem \ref{thm:HEDDot} implies
$$
\HE^{\bigdot\bigdot} = 0.
$$
On the other hand, due to $\sec = 0$ the second formula says
$$
\HE^{\bigdot\bigdot} = \int_{M_r} 2h \det \dot B \, \darea.
$$
Applying Corollary \ref{cor:DetDotBNeg}, we obtain $\dot B = 0$.

It suffices to show that $\dot B = 0$ implies triviality of the deformation $\xi$. Here $\dot B$ is understood as the variation of the shape operator of the surface
$$
M_r \subset (\R_+ \times \Sph^2, \wg_r).
$$
But, similar to the first part of the proof of Theorem \ref{thm:RedToR}, this is the same as the variation of the shape operator of the embedding
$$
\phi_t = \phi_0 + t \xi \colon \Sph^2 \to \R^3.
$$
With this interpretation of $\dot B$, triviality of an infinitesimal deformation that preserves both the metric and the shape operator is well known. This can be seen as an infinitesimal version of the uniqueness part of the Bonnet theorem: first and second fundamental forms determine a surface uniquely. An alternative proof is to use the relation between $\dot B$ and the differential $d\eta$ of the rotation field, see Lemmas \ref{lem:ConstRot} and \ref{lem:DEtaBDot}.

Thus every isometric infinitesimal deformation is trivial, and the theorem is proved.
\end{proof}

\section{Connections between Gauss and metric rigidity}
\subsection{Shearing vs. bending}
\label{subsec:ShearBend}
The proof of Theorem \ref{thm:GaussRig} in Subsection \ref{subsec:ProofGaussRig} and Blaschke's proof of Theorem \ref{thm:InfRigSm} in Subsection \ref{subsec:BlaschkeProof} are almost identical, although in the former $\eta$ is an isogauss infinitesimal deformation while in the latter $\eta$ is the rotation field of an isometric infinitesimal deformation. This similarity is explained by the following direct connection between Gauss and metric infinitesimal rigidity of surfaces in $\R^3$.

\begin{lem}
\label{lem:ShearBend}
Let $M \subset \R^3$ be a smooth surface.
\begin{enumerate}
\item If $\xi \colon M \to \R^3$ is an isometric infinitesimal deformation of $M$, then its rotation vector field $\eta$ is an isogauss infinitesimal deformation of~$M$.
\item Conversely, if $\eta$ is an isogauss infinitesimal deformation of $M$ and $H^1(M) = 0$, then there exists an isometric infinitesimal deformation $\xi$ of $M$ with rotation vector field $\eta$.
\end{enumerate}
In both cases, $\xi$ is the restriction of a Killing vector field if and only if $\eta$ is constant.
\end{lem}
\begin{proof}
Let $\xi$ be an isometric infinitesimal deformation. By Lemma \ref{lem:DEtaTM}, its rotation vector field satisfies $d\eta(X) \in T_xM$ for all $X \in T_xM$. By \eqref{eqn:TrEta}, we also have $\tr (d\eta) = 0$ which means that $\eta$, viewed as an infinitesimal deformation of $M$, preserves the Gauss curvature in the first order, see Lemma \ref{lem:VarK}. Thus both conditions in Defintion \ref{dfn:Isogauss} are satisfied and the first part of the lemma is proved.

In the opposite direction, assume that $\eta$ is an isogauss infinitesimal deformation. Then we have $\tr(d\eta) = 0$. This implies that the vector-valued 1--form $\eta \times dp$ is closed. Since $H^1(M) = 0$ by assumption, there exists a vector field $\xi$ along $M$ such that
$d\xi = \eta \times dp$. Then, clearly, $\langle dp, d\xi \rangle = 0$, so that $\xi$ is an isometric infinitesimal deformation of $M$. By construction, $\eta$ is its rotation field, and the second part of the lemma is proved.

The last statement of the lemma is contained in Lemma \ref{lem:ConstRot}.
\end{proof}

Roughly speaking, it is the correspondence between Theorems \ref{thm:GaussRig} and \ref{thm:InfRigSm} described in the above Lemma that lead Blaschke to his proof. The matter is a bit complicated by the fact that Theorem \ref{thm:GaussRig} seems to not have been explicitely stated, neither before Blaschke nor by himself. Here is a detailed account of events.

In the note \cite{Bla12} Blaschke observed that the rotation vector field of an isometric infinitesimal deformation satisfied a certain differential equation which, as Hilbert had shown in \cite[Section XIX]{Hil10}, had only constant solutions. At the end of the paper, Blaschke promised to expand his argument and provide a geometric interpretation in a later article, but no such article appeared in the following years. Then Weyl \cite{Weyl17} elaborated Blaschke's argument, including also the discrete case: infinitesimal rigidity of convex polyhedra. Blaschke rewrote this proof in a concise form in \cite{Bla21}; since then it became a classical argument and can be found i.~a. in \cite[Chapter 12]{SpiV}.

Hilbert dealt in \cite[Section XIX]{Hil10} with Minkowski's theory of mixed volumes which involves deforming a convex surface by parallelly translating its tangent planes. A lemma in Hilbert's work can be interpreted as the infinitesimal rigidity statement in Theorem \ref{thm:GaussRig}. Minkowski worked mainly with convex polyhedra, and an analogous infinitesimal rigidity statement is Theorem 2.1 in \cite{Izm11a}. See \cite[Subsection 4.5]{Izm11a} also for a discrete analog of Lemma \ref{lem:ShearBend}.

To explain the title of this Subsection, note that an isogauss infinitesimal deformation is shearing without bending while an isometric deformation is bending without shearing, and Lemma \ref{lem:ShearBend} transforms bending into shearing.

\subsection{Polar duality between Gauss and metric rigidity}
\label{subsec:PolDual}
There is another connection between Gauss and metric infinitesimal rigidity. It relates metric rigidity of a surface with Gauss rigidity of its polar dual, as opposed to the previous subsection, where everything happens on a single surface.

It will be convenient for us to change the setup and consider a surface $M$ as parametrized by a map $\phi \colon S \to \R^3$, where $S$ is an abstract smooth surface. This extends the scope a little, as we are able to consider immersed surfaces instead of embedded ones.

Everywhere in this subsection we assume the surface $S$ to be orientable. Let $\nu \colon S \to \R^3$ be a field of unit normals to $M = \phi(S)$. Recall that $h(x) = \langle \nu, \phi(x) \rangle$ is the support function of the immersion $\phi$. We don't distinguish the map $\phi$ and the position vector $p$ in this Subsection.

\begin{dfn}
\label{dfn:PolDual}
Assume that the support function of immersion $\phi \colon S \to \R^3$ nowhere vanishes. Then the map
\begin{gather*}
\psi \colon S \to \R^3,\\
\psi(x) = \frac{\nu(x)}{h(x)},
\end{gather*}
is called the \emph{polar dual} of $\phi$.

Geometrically, $\psi(x)$ is the pole of the plane $d\phi(T_xS)$ with respect to the unit sphere in $\R^3$ centered at the origin.
\end{dfn}

The map $\psi$ is always smooth, but may fail to be an immersion. For example, if $\phi$ maps an open subset $U \subset S$ into a plane, then $\psi$ maps all of $U$ to the pole of this plane.

\begin{lem}
Let $\psi$ be the polar dual of an immersion $\phi$. If $\psi$ is itself an immersion, then the polar dual of $\psi$ is $\phi$.
\end{lem}
\begin{proof}
The condition on $\psi$ in Definition \ref{dfn:PolDual} is equivalent to
$$
\langle \phi, \psi \rangle = 1, \quad \langle d\phi, \psi \rangle = 0.
$$
This determines the map $\psi$ uniquely up to a multiplication with $-1$ that corresponds to inverting the field of unit normals to $\phi$.

By taking the differential of the first equation and subtracting the second one, we obtain $\langle d\psi, \phi \rangle = 0$. Together with the first equation (and under assumption $\rk d\psi = 2$) this forms the conditions on $\phi$ being the polar dual of $\psi$, under an appropriate choice of the field of unit normals to $\psi$.
\end{proof}

\begin{exl}
If $\phi \colon \Sph^2 \to \R^3$ is a smooth embedding with everywhere positive Gauss curvature and such that $0$ lies in the interior of the body $P$ bounded by $\phi(\Sph^2)$, then the polar dual $\psi$ also enjoys all of these properties. If $\nu$ in Definition \ref{dfn:PolDual} is the outward normal, then the body $Q$ bounded by $\psi(\Sph^2)$ can be described as
$$
Q = \{w \in \R^3\, |\, \langle v, w \rangle \le 1 \mbox{ for all } v \in P\}.
$$
This is sometimes used as a definition of the polar dual of a convex body.
\end{exl}

Now we are ready to establish the announced equivalence.

\begin{lem}
\label{lem:PolDualDeform}
Let $\phi \colon S \to \R^3$ be an immersion such that its polar dual $\psi \colon S \to \R^3$ is also an immersion.
\begin{enumerate}
\item \label{it:XiToTau} If $\xi \colon S \to \R^3$ is an isometric infinitesimal deformation of $M = \phi(S)$, then its translation vector field $\tau$ is an isogauss infinitesimal deformation of $N = \psi(S)$.
\item Conversely, if $\tau$ is an isogauss infinitesimal deformation of $N$ and $H^1(S) = 0$, then there is an isometric infinitesimal deformation $\xi$ of $M$ such that $\tau$ is its rotation vector field.
\end{enumerate}
In both cases, $\xi$ is induced by a Killing field on $\R^3$ if and only if $\tau$ is constant.
\end{lem}
\begin{proof}
The part \eqref{it:XiToTau} is essentially proved in Lemma \ref{lem:Tau1}. There we have shown that $\langle d\tau, \phi \rangle = 0$ and $\det(d\psi, d\tau) = 0$. The former is equivalent to $d\tau(X) \parallel TN$, as $p$ is orthogonal to $TN$, and together with the latter implies that the infinitesimal deformation $\tau$ of $N$ preserves Gauss curvature in the first order, see Subsection \ref{subsec:VarGaussGauss}. Thus both conditions of Definition \ref{dfn:Isogauss} are fulfilled and $\tau$ is an isogauss infinitesimal deformation of $N = \psi(S)$.

In the opposite direction, assume that $\tau \colon S \to \R^3$ is given such that $\langle d\tau, \phi \rangle = 0$ and $\det(d\psi, d\tau) = 0$. It follows that the 1--form $d\tau \times \psi$ is closed, and due to $H^1(S) = 0$ there is a vector field $\eta \colon S \to \R^3$ such that $d\eta = d\tau \times \psi$. It follows that
$$
\phi \times d\eta = \phi \times (d\tau \times \psi) = \langle \phi, \psi \rangle d\tau - \langle \phi, d\tau \rangle \psi = d\tau.
$$
Thus if we put $\xi = \eta \times \phi + \tau$, then $\xi$ is an isometric infinitesimal deformation of $M = \phi(S)$.

The translation vector field of an isometric infinitesimal deformation is constant by definition. If $\tau$ is constant, then the equation $d\eta = d\tau \times \psi$ shown above implies that $\eta$ is also constant, and therefore $\xi$ is trivial. The lemma is proved.
\end{proof}

\subsection{Darboux wreath}
\label{subsec:DarbouxWreath}
By combining Lemmas \ref{lem:ShearBend} and \ref{lem:PolDualDeform} one can find another correspondence which is involutive up to sign.

\begin{lem}
\label{lem:PolDualRig}
Let $\phi \colon S \to \R^3$ be an immersion such that its polar dual $\psi \colon S \to \R^3$ is also an immersion. Assume that $\xi$ is an isometric infinitesimal deformation of $M = \phi(S)$ with associated rotation and translation vector fields $\eta$ and $\tau$. Then the vector field
\begin{equation}
\label{eqn:Zeta}
\zeta = \tau \times \psi - \eta
\end{equation}
is an isometric infinitesimal deformation of $N = \psi(S)$ with $\tau$ and $\eta$ as associated rotation and translation vector fields.
\end{lem}
\begin{proof}
By Lemma \ref{lem:ShearBend}, $\tau$ is an isogauss infinitesimal deformation of $\psi$. By Lemma \ref{lem:PolDualDeform}, there exists an isometric infinitesimal deformation $\zeta$ of $\psi$ with rotation vector field $\tau$. To find the corresponding translation vector field, compute
\begin{equation}
\label{eqn:PsiDTau}
\psi \times d\tau = \psi \times (\phi \times d\eta) = \langle \psi, d\eta \rangle \phi - \langle \psi, \phi \rangle d\eta = - d\eta.
\end{equation}
Thus by Definition in Subsection \ref{subsec:RotTranslFields} $-\eta$ is the translation vector field of $\zeta$. The lemma is proved.
\end{proof}

Assume that all maps from $S$ to $\R^3$ under consideration are immersions. Then $\tau$ is an isometric infinitesimal deformation of $\eta$ with rotation vector field $\phi$ and translation vector field $\xi$. In other words, the map
$$
\mbox{(surface, deformation)} \mapsto \mbox{(rotation field, translation field)}
$$
is an involution. On the other hand, one can view $\eta$ as an isometric infinitesimal deformation of $\tau$. Then, by equations \eqref{eqn:Zeta} and \eqref{eqn:PsiDTau}, the corresponding rotation and translation vector fields are $-\psi$ and $-\zeta$. That is, the map
\begin{equation}
\label{eqn:Wreath}
\mbox{(surface, deformation)} \mapsto \mbox{(translation field, rotation field)}
\end{equation}
has the orbit $(\phi, \xi) \mapsto (\tau, \eta) \mapsto (-\zeta, -\psi) \mapsto \cdots$. By Darboux, \cite{Dar96}, \cite[Section 3.4.1]{Sab92}, the map \eqref{eqn:Wreath} has order six in general. The twelve surfaces in the orbit are called the \emph{Darboux wreath}.

\section{Polar duality between volume and Hilbert-Einstein functional}
\subsection{A conjecture about second derivatives}
Lemma \ref{lem:PolDualDeform} establishes a correspondence between isometric infinitesimal deformations of a surface $M$ and isogauss infinitesimal deformations of its polar dual $N$. In the case of convex closed surfaces, we identified isometric deformations with zeros of the second derivative $\HE^\bigdot\bigdot$ of the Hilbert-Einstein functional, and isogauss deformations with zeros of the second derivative $\Vol^\bigdot\bigdot$ of the volume. This suggests that there is a relation between the Hilbert-Einstein functional of a convex body $P$ and the volume of the polar dual $P^*$, at least on the level of second derivatives.

Let $M \subset \R^3$ be a convex closed surface with positive Gauss curvature and enclosing the coordinate origin, and let $N$ be its polar dual. Denote by $r_0 \colon \Sph^2 \to \R$ the distance function of $M$ (see Definition \ref{dfn:DistFunc}, and by $h_0 \colon \Sph^2 \to \R$ the support function of $N$ (see Subsection \ref{subsec:HessSupp}). Clearly, we have $h_0 = 1/r_0$. By equation \eqref{eqn:HEDot}, we have
$$
\HE^\bigdot = \int\limits_{\Sph^2} \dot r \frac{\sec}{\cos \alpha} \, \darea_g,\
$$
where $g$ is the metric induced on $\Sph^2$ by the radial projection $\phi_0 \colon \Sph^2 \to M$. Denote by $L_1$ the linearization at $r_0$ of the operator
$$
r \mapsto \frac{\sec}{\cos \alpha}.
$$
As the metric $g$ remains constant during the warped product deformations considered in Section \ref{sec:RigHE}, the above formula implies that $L_1$ is self-adjoint. Similarly, we have
$$
\Vol^\bigdot = \int\limits_{\Sph^2} \dot h \, \det(h \id + \Hess h) \, \darea_{\Sph^2},
$$
see Subsection \ref{subsec:HessSupp}. Again, the linearization $L_2$ of the operator
$$
h \mapsto \det(h \id + \Hess h)
$$
at $h = h_0$ is self-adjoint.

\begin{lem}
\label{lem:KerEqual}
$\ker L_1 = \ker L_2$
\end{lem}
\begin{proof}
As $\sec = 0$ at $r = r_0$, we have $\ker L_1 = \{\dot r \, |\, \sec^\bigdot = 0\}$. By Subsection \ref{subsec:RedWarp}, there is an isometric infinitesimal deformation $\xi$ of $M$ such that
$$
\dot r(x) = \langle \xi(x), x \rangle.
$$
By Lemma \ref{lem:PolDualDeform}, the corresponding translation vector field $\tau$ is an isogauss infinitesimal deformation of $N$. As $\det(h \id + \Hess h)$ is the reciprocal of the Gauss curvature of $N$, it follows that the corresponding variation of $h$ belongs to the kernel of $L_2$. Since we have
$$
\dot h(x) = \langle \tau(x), x \rangle = \langle \xi(x), x \rangle = \dot r(x),
$$
where the second equation follows from $\xi(x) = \eta(x) \times x + \tau(x)$, it follows that $\ker L_1 \subset \ker L_2$. By inverting the argument, we obtain $\ker L_2 \subset \ker L_1$, and the lemma is proved.
\end{proof}

The above lemma might seem less exciting, as the rigidity theorems imply that both kernels correspond to trivial deformations. But it says more when understood locally: note that Lemma \ref{lem:PolDualDeform} is of a local character, as are all other arguments in the proof of Lemma \ref{lem:KerEqual}.

\begin{conj}
Operators $L_1$ and $L_2$ are equal.
\end{conj}

This conjecture is motivated by a polyhedral analog, see \cite[Lemma 4.1]{Izm11a}.

The rest of this section deals with spherical and hyperbolic-de Sitter geometry, where the relation between the Hilbert-Einstein functional and volume of the dual is a more straightforward one.

\subsection{Polar duality and fundamental forms in spherical geometry}
\label{subsec:PolDualFundForms}
Let $S$ be an orientable smooth surface and $\phi \colon S \to \Sph^3$ be an immersion. We say that the surface $N = \phi(\Sph^3)$ is \emph{co-oriented}, if for every $x \in S$ one of the half-spaces in which $T_{\phi(x)}N$ divides $T_{\phi(x)}\Sph^3$ is dubbed positive, and this in a continuous way. A co-orientation can be introduced by chosing a unit normal field pointing in the positive direction.

\begin{dfn}
\label{dfn:SphPolDual}
The \emph{polar dual} of a co-oriented immersion $\phi$ is the map $\psi \colon S \to \Sph^3$ that sends $x$ to the pole of the $2$--sphere through $\phi(x)$ tangent to $T_{\phi(x)}M$. Of the two poles the one is chosen that lies on the positive side.
\end{dfn}

\begin{lem}
\label{lem:SphPolPairs}
Let $\psi$ be the polar dual of an immersion $\phi \colon S \to \Sph^3$. If $\psi$ is itself an immersion, then $\phi$ is the polar dual to $\psi$, for an appropriate co-orientation of $N = \psi(S)$.
\end{lem}
\begin{proof}
View $\Sph^3$ as the unit sphere in $\R^4$ centered at the origin. For any $v, w \in \Sph^3$ the scalar product $\langle v, w \rangle$ is then understood as the scalar product of corresponding vectors in $\R^4$. Then Definition \ref{dfn:SphPolDual} implies
$$
\langle \phi, \psi \rangle = 0, \quad \langle d\phi, \psi \rangle = 0.
$$
This determines the map $\psi$ uniquely up to antipodal involution which corresponds to changing the co-orientation of $M = \phi(S)$. Taking the differential of the first equation and subtracting the second one yields $\langle \phi, d\psi \rangle = 0$. Together with the first equation and under assumption $\rk d\psi = 2$ this means that $\phi$ is the polar dual of $\psi$, for a certain co-orientation of $N = \psi(S)$.
\end{proof}

The polar dual can be interpreted as follows: $\psi(x)$ is obtained as the endpoint of the co-orienting unit normal $\nu(x)$ to $\phi$, translated so that its starting point is at the origin of $\R^4$. It follows that
$$
d\psi(X) = \wnabla_X \nu,
$$
which should again be understood as equality between free vectors in $\R^4$. Therefore the three fundamental forms of the immersion $\phi$ can be written as follows:
\begin{equation*}\begin{split}
\I_\phi(X,Y) &= \langle d\phi(X), d\phi(Y) \rangle ,\\
\II_\phi(X,Y) &= \langle d\phi(X), d\psi(Y) \rangle = \langle d\phi(Y), d\psi(X) \rangle,\\
\III_\phi(X,Y) &= \langle d\psi(X), d\psi(Y) \rangle.
\end{split}\end{equation*}
This immediately implies the following lemma.

\begin{lem}
\label{lem:SphDualFundForms}
Assume that the polar dual $\psi$ of an immersion $\phi$ is itself an immersion. Interpreting $\phi$ as a unit normal field along $\psi$, co-orient $\psi$ by $\phi$. Then we have
\begin{gather*}
\I_\phi = \III_\psi,\\
\II_\phi = \II_\psi,\\
\III_\phi = \I_\psi.
\end{gather*}
\end{lem}

Finally, let us characterize immersions whose polar duals are also immersions.

\begin{lem}
The polar dual $\psi$ of an immersion $\phi$ has full rank at $x \in S$ if and only if $\phi(x)$ is not a parabolic point for $\phi$.
\end{lem}
\begin{proof}
Choose an arbitrary volume form on $S$. This gives sense to determinants $\det \I_\phi$, $\det \II_\phi$, and $\det \III_\phi$. Then we have
$$
\frac{\det \III_\phi}{\det \I_\phi} = \det (B^2) = (\det B)^2,
$$
where $B$ is the shape operator of the surface $M = \phi(S)$. Thus the symmetric bilinear form $\I_\psi = \III_\phi$ is non-degenerate if and only if $\det B \ne 0$. The lemma is proved.
\end{proof}

In particular, if $K(x) > 1$ is the Gauss curvature of $\phi$, then $\psi$ is an immersion with Gauss curvature $K(x)^{-1}$.

\subsection{Polar duality between Gauss and metric rigidity for surfaces in $\Sph^3$}
\label{subsec:PolDualRig}
Let $\psi \colon S \to \Sph^3$ be an immersion. An infinitesimal deformation of $\psi$ is a vector field along $\psi$:
\begin{gather*}
\eta \colon S \to T\Sph^3,\\
\eta(x) \in T_{\psi(x)} \Sph^3.
\end{gather*}
The geodesic flow of $\psi$ along $\eta$ defines a family of immersions (for small $t$)
\begin{gather*}
\psi_t \colon S \to \Sph^3,\\
\psi_t(x) = \exp_{\psi(x)}(t\eta(x)).
\end{gather*}
Clearly, $\psi_0 = \psi$. Denote
$$
\dot{\III}_\psi = \left.\frac{d}{dt} \right|_{t=0} \III_{\psi_t}.
$$

The following definition is known to be the spherical analog of the Gauss infinitesimal rigidity defined in Subsection \ref{subsec:DefAndThm}.

\begin{dfn}
A vector field $\eta$ along an immersion $\psi$ is called an \emph{isogauss} infinitesimal deformation of $N$, if $\dot{\III}_\psi = 0$. The surface $N = \psi(S)$ is called \emph{Gauss infinitesimally rigid} if every its isogauss infinitesimal deformation is trivial, i.~e. is a restriction of a Killing field on $\Sph^3$.
\end{dfn}

In $\Sph^3$, Gauss infinitesimal rigidity is straightforwardly related to the metric infinitesimal rigidity of the polar dual.

\begin{lem}
\label{lem:SphRigEquiv}
Let $(\phi, \psi)$ be a polar pair of immersions of an orientable surface $S$ in $\Sph^3$. Then the surface $N = \psi(S)$ is Gauss infinitesimally rigid if and only if the surface $M = \phi(S)$ is metrically infinitesimally rigid.
\end{lem}
\begin{proof}
Let $\eta$ be an infinitesimal deformation of $q$ and $\psi_t = \exp_q(t\eta)$ be the corresponding
geodesic flow. Let $\phi'_t$ be the polar dual of $\psi_t$. By Lemma \ref{lem:SphDualFundForms}, we have $\I_{\phi'_t} = \III_{\psi_t}$, hence
$$
\left.\frac{d}{dt} \right|_{t=0} \I_{\phi'_t} = \dot{\III}_\psi.
$$
The left hand side depends only on the 1--jet of $\phi'_t$, hence it does not change if we replace $\phi'_t$ by $\phi_t = \exp_\phi(t\xi)$, where $\xi(x) = d/dt|_{t=0} \phi_t(x)$. It follows that $\xi$ is an isometric infinitesimal deformation of $\phi$ if and only if $\eta$ is an isogauss infinitesimal deformation of $\psi$.

If $\eta$ is induced by a Killing vector field on $\Sph^3$, then $\psi_t = \Psi_t \circ q$ for the corresponding one-parameter group of isometries $\Phi_t$. Then also $\phi_t = \Psi_t \circ p$. It follows that $\xi$ is induced by the same Killing vector field and is therefore a trivial deformation. The lemma is proved.
\end{proof}

\subsection{Herglotz's formula in $\Sph^3$}
\label{subsec:HEPolDual}
The last subsection relates Gauss rigidity and metric rigidity of the polar dual in $\Sph^3$ in a very natural way. Here we recall a formula of Herglotz that relates the Hilbert-Einstein functional to the volume of the polar dual. In the next subsection we discuss corresponding variational approaches to infinitesimal rigidity of surfaces in $\Sph^3$.

\begin{thm}[Herglotz]
\label{thm:SphWeylTube}
Let $P \subset \Sph^3$ be a convex body with smooth boundary $M$ with Gauss curvature bigger than $1$. Let $P^* \subset \Sph^3$ be the body bounded by the surface $M^*$ polar dual to $M$. Then we have
\begin{equation}
\label{eqn:SphWeylTube}
\Vol(P) + \int_M H\, \darea + \Vol(P^*) = \pi^2,
\end{equation}
where $H$ is the mean curvature, i.~e. half the trace of the shape operator.
\end{thm}
\begin{proof}
Assume $M$ to be co-oriented by the outward unit normal $\nu$, so that the bodies $P$ and $P^*$ are disjoint (like northern and southern ice caps). We will compute the volume of the complement $\Sph^3 \setminus (P \cup P^*)$.

The field of unit normals to $M$ can be viewed as a map $\nu \colon M \to M^*$, and the differential $d\nu$ can be identified with the shape operator $B$ on $M$, cf. Subsection \ref{subsec:PolDualFundForms}. Consider the map
$$
\begin{array}{l}
F \colon M \times [0, \frac{\pi}{2}] \to \Sph^3,\\
F(x,t) = \cos t \cdot x + \sin t \cdot \nu(x).
\end{array}
$$
Note that $F(M \times \{t\})$ is the set of points at distance $t$ from $M$ and distance $\frac{\pi}{2} - t$ from $M^*$. Thus the map $F$ is a diffeomorphism onto the closure of $\Sph^3 \setminus (P \cup P^*)$ and we have
$$
\Vol(\Sph^3 \setminus (P \cup P^*)) = \int\limits_{M \times [0, \frac{\pi}{2}]} F^*(\dvol_{\Sph^3}).
$$
By denoting $F_t = F(\cdot\,, t)$, we can write
$$
F^*(\dvol_{\Sph^3}) = dt \, F_t^*(\darea_{F_t(M)}) = \det(dF_t) \, dt \darea_M.
$$
As $dF_t = \cos t \cdot \id + \sin t \cdot B$, we have
$$
\det(dF_t) = \cos^2 t + 2 \sin t \cos t \cdot H + \sin^2 t \cdot K.
$$
It follows that
\begin{equation*}\begin{split}
\Vol(\Sph^3 \setminus (P \cup P^*)) &= \int\limits_M \darea \int\limits_0^{\frac{\pi}{2}} (\cos^2 t + 2 \sin t \cos t \cdot H + \sin^2 t \cdot K) dt\\
&= \int\limits_M \left( \frac\pi4 + H + \frac{\pi}4 K \right)\, \darea\\
&= \frac{\pi}{4}\Area(M) + \int\limits_M H\, \darea + \frac{\pi}{4}\Area(M^*).
\end{split}\end{equation*}
As $\Vol(\Sph^3) = 2\pi^2$, this leads to
\begin{equation}
\label{eqn:HMcM+}
\Vol(P) + \frac{\pi}{4}\Area(M) + \int\limits_M H\, \darea + \frac{\pi}{4}\Area(M^*) + \Vol(P^*) = 2\pi^2
\end{equation}
(an instance of Steiner's formula on $\Sph^3$). On the other hand, a simple computation:
\begin{multline*}
\Area(M^*) = \int\limits_M \det(dF_{\frac{\pi}{2}}) \, \darea_M = \int\limits_M \det B \, \darea_M\\
= \int\limits_M (K-1) \, \darea_M = 4\pi - \Area(M)
\end{multline*}
proves the formula
\begin{equation}
\label{eqn:SphGB}
\Area(M) + \Area(M^*) = 4\pi.
\end{equation}
By multiplying it with $\pi/4$ and subtracting from \eqref{eqn:HMcM+} we obtain \eqref{eqn:SphWeylTube}. The theorem is proved.
\end{proof}

\begin{rem}
There is an alternative to applying the Gauss-Bonnet formula in the second part of the proof. Consider the map
$$
\begin{array}{l}
G \colon M \times [0, \frac{\pi}{2}] \to \Sph^3,\\
G(x,t) = \cos t \cdot x - \sin t \cdot \nu(x).
\end{array}
$$
Then $G_0$ maps $M$ identically to itself, whereas $G_{\pi/2}$ maps $M$ to $-M^*$. Attach a ball to  each of the bases of the cylinder $M \times [0, \frac{\pi}{2}]$ and extend the map $G$ by mapping these balls to $P$ and $-P^*$, respectively. This results in a piecewise smooth map $\bar G \colon \Sph^3 \to \Sph^3$ (the source space itself is equipped with only piecewise smooth structure). As $\bar G$ has degree $0$, we have
$$
\Vol(P) + \int\limits_{M \times [0, \frac{\pi}{2}]} G^*(\dvol_{\Sph^3}) + \Vol(P^*) = \int\limits_{\Sph^3} \bar G^*(\dvol_{\Sph^3}) = 0.
$$
By performing computations similar to the first part of the proof, we obtain
\begin{equation}
\label{eqn:HMcM-}
\Vol(P) - \frac{\pi}{4}\Area(M) + \int\limits_M H\, \darea - \frac{\pi}{4}\Area(M^*) + \Vol(P^*) = 0.
\end{equation}
By summing this with \eqref{eqn:HMcM+}, we obtain \eqref{eqn:SphWeylTube}.

The argument that uses maps $F$ and $G$ is a slight modification of Herglotz's proof in \cite{Her43b}. Herglotz works in $\Sph^d$ for arbitrary $d$ and obtains two formulas of which \eqref{eqn:SphWeylTube} and \eqref{eqn:SphGB} are special cases for $d = 3$. For an odd $d$ both formulas are self-dual (in particular, $\int_M H \darea$ is also the total mean curvature of $M^*$), while for an even $d$ they are dual to each other, so that the terms $\Vol(P)$ and $\Vol(P^*)$ occur in different formulas.
\end{rem}

\begin{rem}
Points of $P^*$ are poles of great spheres disjoint with the interior of $P$. This provides an integral-geometric interpretation of the formula \eqref{eqn:SphWeylTube}: a random great sphere intersects a convex body $P$ with the probability
$$
\frac{1}{\pi^2} \left( \Vol(P) + \int\limits_{\partial P} H\, \darea \right).
$$
\end{rem}

\subsection{Approaches to proving infinitesimal rigidity of convex surfaces in $\Sph^3$}
\label{subsec:ApprSphere}
A smooth convex closed surface $M \subset \Sph^3$ with Gauss curvature bigger than $1$ is infinitesimally rigid, see Remark \ref{rem:DarPog} below. Most likely, a direct proof can be found that uses the approximating section in the bundle of germs of Killing fields, similarly to the rotation and translation fields approach in Subsections \ref{subsec:RotTranslFields}--\ref{subsec:OtherProof}. The proof from Subsection \ref{subsec:OtherProof} should be equivalent to studying derivatives of the functional
$$
S(P) := 2\Vol(P) + \frac12 \int\limits_P \scal \, \dvol + 2 \int\limits_{\partial P} H \, \darea,
$$
where $P$ is the body bounded by $M$, and the metric in the interior of $P$ varies in the class of warped products, cf. Section \ref{sec:RigHE}. The functional $S$ can be seen as a gravity action with non-zero cosmological constant, cf. \cite{KS08}.

Under the above assumptions, the surface $M$ is also Gauss infinitesimally rigid. As indicated by Subsections \ref{subsec:PolDualRig} and \ref{subsec:HEPolDual}, a proof of (metric) infinitesimal rigidity of $M$ using the Hilbert-Einstein functional should translate in a straightforward way as a proof of Gauss infinitesimal rigidity of $M^* = \partial P^*$ using the volume of $P^*$. One should probably study the functional
$$
S^*(P^*) := 2\Vol(P^*) + \frac12 \int\limits_{P^*} \scal \, \dvol,
$$
that is the gravity action without the boundary term, where the metric in the interior of $P^*$ varies in the class of warped products while preserving the third fundamental form of the boundary. Indeed, warping the metric around the north pole of $\Sph^3$ that is contained in $P$ is equivalent to warping around the south pole contained in $P^*$, and we conjecture that $S(P) + S^*(P^*) = 2\pi^2$ holds for every warped product metric on $\Sph^3$, cf. \cite[Lemma 4.9]{Izm11a}.

\subsection{Polar duality in hyperbolic-de Sitter geometry}
\label{subsec:HypDeSitter}
Less classical as the polar duality in the sphere is the polar duality between the hyperbolic and de Sitter spaces, see \cite[Section 1]{Scl06} and references therein.

Consider the hyperboloid model of the hyperbolic space
$$
\H^3 = \{ x \in \R^4\ |\ \|x\|_{3,1} = -1, x_0 > 0 \},
$$
where $\|x\|_{3,1} = -x_0^2 + x_1^2 + x_2^2 + x_3^2$. The polar dual to an immersion $\phi \colon S \to \H^3$ is an immersion in the \emph{de Sitter space}
$$
\dS^3 = \{x \in \R^4\ |\ \|x\|_{3,1} = 1\}.
$$
Indeed, if we define the polar dual $\psi$ similarly to Subsection \ref{subsec:PolDualFundForms}  through
$$
\langle \phi, \psi \rangle = 0, \quad \langle d\phi, \psi \rangle = 0,
$$
then $\psi(x)$ lies on a space-like line, which intersects $\dS^3$ and not $\H^3$.

Relations between the fundamental forms and the Gauss and metric infinitesimal rigidity for polar pairs of surfaces carry over to the hyperbolic-de Sitter case. In order to transfer Theorem \ref{thm:SphWeylTube}, one has to give meaning to the term $\Vol(P^*)$. First of all, we put
$$
P^* = \cone(M^*) \cap (\dS^3 \cup \H^3_-),
$$
where $\cone(M^*) \subset \R^4$ is the cone over $M^*$, and $\H^3_-$ is the antipodal copy of $\H^3$. Thus $P^*$ is the union of $\H^3_-$ and of an infinite end of $\dS^3$. However, there is a consistent way to define a finite quantity $\Vol(P^*)$.

\begin{thm}
\label{thm:HypWeylTube}
Let $P \subset \H^3$ be a convex body with smooth boundary $M$ with everywhere positive definite shape operator. Let $P^* \subset \dS^3$ be the convex body bounded by the surface $M^*$ polar dual to $M$. Then we have
\begin{equation}
\label{eqn:HypWeylTube}
\Vol(P) - \int_M H\, \darea + \Vol(P^*) = 0,
\end{equation}
where $H$ is the mean curvature, i.~e. half of the trace of the shape operator.
\end{thm}

Herglotz \cite{Her43b} proves this theorem in the same way as in the spherical case, multiplying the arguments of $\sin$ and $\cos$ with $i$. In footnote 10, he remarks that the same result can be achieved by studying the asymptotics of the volume of parallel bodies in $\H^3$. In \cite{CK06}, contour integrals in the complex plane are used in order to assign a finite volume to certain subsets of $\H^3 \cup \dS^3$. This might be related to the argument of Herglotz.

\begin{rem}
An integral geometric interpretation of \eqref{eqn:HypWeylTube} says that the (motion-invariant and appropriately normalized) measure of the set of planes that intersect a convex body $P \subset \H^3$ equals
$$
\int\limits_{\partial P} H\, \darea - \Vol(P).
$$
In particular, this quantity is always positive.
\end{rem}

\begin{rem}
In \cite{KS08}, the asymptotics of the volume of parallel bodies of the convex core is used to define the renormalized volume of a non-compact hyperbolic manifold.
\end{rem}

\begin{rem}
\label{rem:DarPog}
Infinitesimal rigidity of smooth surfaces is invariant under projective transformations, as was shown by Darboux \cite{Dar96}. Closely related to this are the so-called Pogorelov maps, \cite[Chapter 5]{Pog73}. They associate to an isometric infinitesimal deformation of a surface $M \subset \R^3$ an isometric infinitesimal deformation of the surface $M_\H \subset \H^3$ obtained by taking an open ball $B \supset M$ and interpreting it as the Klein model of $\H^3$. As a result, $M_\H$ is infinitesimally rigid if and only if $M$ is. The same holds with $M_\H$ replaced by $M_\Sph \subset \Sph^3$ defined as the inverse gnomonic projection of $M$. Volkov in \cite{Vol74} gives a unified treatment of these two results.

A surface $M \subset \R^3$ is convex if and only if $M_\H$ (respectively, $M_\Sph$) is convex. Thus infinitesimal rigidity of closed strictly convex surfaces in the hyperbolic (respectively, spherical) space follows from the rigidity in the Euclidean space.
\end{rem}

\section{Relation to Minkowski and Weyl problems}
\subsection{A brief overview}
The following problem was posed by Weyl in \cite{Weyl16}.
\begin{weylpr}
Let $g$ be a Riemannian metric on $\Sph^2$ with everywhere positive Gauss curvature. Show that there exists a smooth convex embedding $\phi \colon \Sph^2 \to \R^3$ such that $\phi^*(\can_{\R^3}) = g$. Show that this embedding is unique up to an isometry of $\R^3$.
\end{weylpr}
Weyl outlined a proof for the analytic case; it was accomplished later by H.~Lewy \cite{Lewy38a}. Nirenberg \cite{Nir53} extended Weyl's method to certain finite differentiability classes.

A.~D.~Alexandrov \cite{Ale42, Ale05} stated and proved a polyhedral analog of the Weyl problem. In this case $g$ is a Euclidean metric with cone points of angles less than $2\pi$, and $(\Sph^2, g)$ must be embedded as a convex polyhedron. By approximating a Riemannian metric with polyhedral ones, Alexandrov showed that $(\Sph^2, g)$ from the Weyl problem can be embedded isometrically into $\R^3$ as a convex surface, but he hasn't shown that the embedding is smooth. Pogorelov filled this gap in \cite{Pog51}, again for some finite differentiability classes, and strengthened his results in later works.

\begin{minkpr}
Let $K \colon \Sph^2 \to (0, +\infty)$ be a smooth function such that
$$
\int\limits_{\Sph^2} K(x) \, x \, \darea_{\Sph^2} = 0,
$$
where $\Sph^2 = \{ x \in \R^3 \, | \, \|x\| = 1 \}$. Show that there exists a smooth convex embedding $\psi \colon \Sph^2 \to \R^3$ such that $K(x)$ is the Gauss curvature of the surface $N = \psi(\Sph^2)$ at the point $\psi(x)$, and $x$ is the outward unit normal to $N$ at $\psi(x)$. Show that the embedding $\psi$ is unique up to a parallel translation.
\end{minkpr}
The problem that was stated and proved by Minkowski in \cite{Min03} is different: he assumed only continuity of $K$ and wanted to prove the existence of a convex surface whose curvature measure has density $K$, the curvature measure being defined as the measure of the Gauss image pulled back to $N$. Minkowski first proved an analog for convex polyhedra and then used polyhedral approximation. Later, the argument was extended to arbitrary measures, not necessarily having a positive continuous density, \cite{Ale38III}, \cite[Section 7.1]{Scn93}. Pogorelov \cite{Pog52} proved that if $K$ is smooth then the corresponding surface is also smooth and thus has Gauss curvature $K$. This settled the Minkowski problem in the formulation given above.

In a different line of research, H.~Lewy \cite{Lewy38b} solved the Minkowski problem for analytic metrics by adapting the method suggested by Weyl for solution of his problem, and Nirenberg \cite{Nir53} extended this to the smooth case.

\subsection{Uniqueness in the Minkowski and Weyl problems}
\label{subsec:MWUniq}
The uniqueness part in both Minkowski and Weyl problems is essentially easier than the existence part. For the former, it was proved by Minkowski himself using his mixed volumes theory. Minkowski's ideas were developed by Alexandrov, who gave another uniqueness proof using mixed determinants, \cite{Ale38IV}. The uniqueness in the Weyl problem was proved by Herglotz in \cite{Her43a}; this proof is reproduced in \cite{Hopf89}.

A proof similar to that of Alexandrov was found by Chern \cite{Che57} and, independently, by Hsiung \cite{Hsiu58}. The main difference is that they argued in terms of the position vector of an embedding while Alexandrov argued in terms of the support function. The proof by Chern and Hsiung is reproduced in \cite[Chapter 12]{SpiV}. One should also mention a paper \cite{Che59} of Chern, where he generalizes his arguments and establishes a connection with Alexandrov's work.

Quite curiously, the ideas came full circle. Hsiung' work was motivated by that of Herglotz. Herglotz refers to Blaschke's proof of the infinitesimal rigidity of convex surfaces. And Blaschke's proof was inspired by Hilbert's work \cite{Hil10} related to the Minkowski problem, see discussion at the end of our Subsection \ref{subsec:ShearBend}.

Our proof of the infinitesimal Gauss rigidity in Subsection \ref{subsec:ProofGaussRig} is related to Chern-Hsiung's proof of the Minkowski uniqueness in the same way as Blaschke's proof of the infinitesimal metric rigidity is related to Herglotz's proof of the Weyl uniqueness. However, we came to our proof by ``smoothing'' the polyhedral analog, \cite{Izm11a}.

\subsection{Existence in the Minkowski and Weyl problems}
Weyl, Lewy, and Nirenberg solved the Minkowski and Weyl problems using the continuity method. For example, in the case of Weyl's problem, they consider a family of metrics $\{g_t \, | \, t \in [0,1]\}$ on $\Sph^2$ such that $g_1 = g$ and $g_0$ is the standard metric of curvature $1$. For $t=0$ there is an isometric embedding $(\Sph^2, g_t) \to \R^3$, and one wants to show that the set of all $t$ for which this is the case is an open and closed subset of $[0,1]$. The openness follows from the ellipticity of a certain operator (which is, in fact, related to infinitesimal rigidity), and closedness follows from a priori estimates.

Alexandrov's solution of the polyhedral Weyl problem also uses a variant of the continuity method. At the same time, Minkowski proved the polyhedral Minkowski theorem in a different, quite elegant way. On the space of convex polyhedra with given directions of outward normals, he maximized the volume function under a certain linear constraint and showed that the maximum point is the desired polyhedron up to a scaling. Alexandrov \cite{Ale38III} extended this method to arbitrary curvature measures.

A variational approach to the Weyl problem was suggested by Blaschke and Herglotz in \cite{BH37}. Their idea was to consider it not as an embedding problem $(\Sph^2, g) \to (\R^3, \can)$, but as an extension problem: given a ball $\Ball^3$, a Riemannian metric $g$ on $\Sph^2 = \partial\Ball^3$ must be extended to a flat metric on $\Ball^3$. On the space of all extensions $\widetilde g$ they consider the Hilbert-Einstein functional and show that its critical points are exactly the flat metrics. Thus the Weyl problem reduces to showing that the Hilbert-Einstein functional has exactly one critical point. This idea was never implemented.

\subsection{An approach to the Minkowski and Weyl problems}
\label{subsec:NewMinkWeyl}
Section \ref{sec:RigHE} of the present paper suggests a modification of Blaschke-Herglotz's approach: instead of considering all Riemannian metrics $\wg$ on $\Ball^3$ that extend a given metric $g$ on $\Sph^2$, one considers only warped products $\wg_r$ of the form \eqref{eqn:TildeGR}. The metric $\wg_r$ is determined by a function $r \colon \Sph^2 \to \R$ and determines a function $\sec \colon \Sph^2 \to \R$ such that $\wg_r$ is flat if and only if $\sec \equiv 0$. We have shown in Section \ref{sec:RigHE} that the kernel of the map $\dot r \mapsto \sec^\bigdot$ (the linearization of $r \mapsto \sec$) at $r$ such that $\wg_r$ is flat has dimension three.

\begin{conj}
The map $\dot r \mapsto \sec^\bigdot$ is onto provided that
\begin{equation}
\label{eqn:BoundSec}
0 < \sec < K,
\end{equation}
where $K \colon \Sph^2 \to \R$ is the Gauss curvature of the metric $g$.
\end{conj}
(Note that $\sec < K$ implies $\det B > 0$, where $B$ is the shape operator on $\Sph^2$ embedded in the warped product.)

If this conjecture is true, then under assumption \eqref{eqn:BoundSec} every small perturbation of $\sec$ can be achieved by an appropriate modification of $r$.

In order to achieve $\sec = 0$, put $r_0 \equiv R$ for a sufficiently large $R \in \R$. Then the corresponding $\sec_0$ satisfies \eqref{eqn:BoundSec}. Construct a family of functions $\{r_t \, |\, t \in [0,1) \}$ by requiring
\begin{equation}
\sec_t = (1-t) \sec_0.
\end{equation}
One needs some a priori estimates in order to show that the set of $t \in [0,1)$ for which $r_t$ exists is closed and that the limit at $t \to 1$ is a smooth function. This would prove the existence in the Weyl problem.

In the polyhedral case, the above approach was realized in \cite{BI08}. A different modification of the Blaschke-Herglotz approach in the polyhedral case was found and realized by Alexandrov's student Volkov in \cite{Vol55}.

A similar approach to the existence part of the Minkowski problem consists in choosing an arbitrary function $h_0 \colon \Sph^2 \to \R$ and constructing a family of functions $\{h_t \, |\, t \in [0,1] \}$ such that $K_t^{-1} - K^{-1} = (1-t)(K_0^{-1} - K^{-1})$, where $K_t$ is the Gauss curvature of the surface with support function $h_t$.

In \cite{ChW00}, the curvature flow $\dot h = - \log(K_t/K)$ is considered. It is shown that, depending on the initial data, the solution either shrinks to a point or expands to infinity or converges to a smooth surface with Gauss curvature~$K$. The advantage of this approach is the expliciteness of the evolution equation, while our approach would give convergence to the solution of the Minkowski problem independently of the initial data.

\section{Directions for future research}
\label{sec:Directions}
\begin{prb}
What is the analog of Darboux wreath (Subsection \ref{subsec:DarbouxWreath}) for infinitesimal deformations of surfaces in the sphere and in the hyperbolic-de Sitter space?
\end{prb}

\begin{prb}
Find new proofs of the existence part in the Minkowski and Weyl problems based on the approach outlined in Subsection \ref{subsec:NewMinkWeyl}.
\end{prb}

Koiso in \cite{Koi78} proved infinitesimal rigidity of Einstein manifolds under certain restrictions on the curvature by studying the second variation of the Hilbert-Einstein functional. Our approach to the infinitesimal rigidity of convex surfaces in Section \ref{sec:RigHE} has some similar features.

\begin{prb}
\label{prb:KoisoBdry}
Prove infinitesimal rigidity of an Einstein manifold $M$ with convex boundary by unifying Koiso's and our arguments.

Koiso used certain Weitzenb\"ock-type formulas for $\HE^{\bigdot\bigdot}$. In the case with boundary an integral over $\partial M$ should appear. Hopefully it can be identified
as $\int_{\partial M} f \det \dot B \, \dvol$ for some positive function $f$.
\end{prb}

Koiso's method works in particular for compact closed hyperbolic manifods (whose infinitesimal rigidity was previously proved by Calabi \cite{Cal61} and Weil \cite{Wei60}). Thus a special case of Problem \ref{prb:KoisoBdry} is to give a new proof of Schlenker's theorem \cite{Scl06} on the infinitesimal rigidity of compact hyperbolic manifolds with convex boundary. Schlenker proves also Gauss infinitesimal rigidity of such manifolds. It would also be interesting to find a variational proof of this theorem, cf. Subsection \ref{subsec:ApprSphere}.

Similar problems for manifolds with convex polyhedral boundary are not solved, although infinitesimal rigidity of convex polyhedra can be proved in a similar spirit, cf. \cite{Izm11a}. A challenge would be to find a common generalization from polyhedral and smooth to arbitrary convex surfaces (respectively, manifolds with convex boundary).
 
\begin{prb}
By extending the approach proposed in Problem \ref{prb:KoisoBdry}, prove the infinitesimal rigidity of compact hyperbolic manifolds with convex boundary which is neither smooth nor polyhedral. In particular, prove the infinitesimal rigidity of the convex core of a hyperbolic manifold.
\end{prb}

Infinitesimal rigidity of convex cores would probably allow to prove the uniqueness part in Thurston's pleating lamination conjecture. See \cite{BO04}, where the existence part is proved.

Note that the Schl\"afli formula for convex cores proved by Bonahon \cite{Bon98} should be equivalent to a formula for the first variation of the Hilbert-Einstein functional under arbitrary variations of the metric in the interior of the convex core. More generally, first variations of Lipschitz-Killing curvatures in the case of non-smooth boundary are computed by Bernig in \cite{Ber03}.

\begin{prb}
In the same spirit, reprove Pogorelov's theorem \cite[Chapter IV]{Pog73} on the infinitesimal rigidity of arbitrary convex surfaces without flat pieces by studying variations of the Hilbert-Einstein functional.
\end{prb}

\appendix

\section{Mixed determinants of linear operators and of vector-valued differential forms}
\label{sec:DetForms}
\subsection{Mixed determinants of linear operators}
\label{subsec:MixDetOp}
Let $V$ be an vector space with $\dim V = n$. Then the determinant
$$
\det \colon \End(V) \to \R
$$
is a homogeneous polynomial of degree $n$. Therefore there is a unique polarization of $\det$, that is, a symmetric polylinear form on $\End(V)$ denoted also by $\det$ and such that
$$
\det(A, \ldots, A) = \det A
$$
for all $A \in \End(V)$. See e.~g. \cite[Appendix A]{Hoer07}.

\begin{dfn}
The number $\det(A_1, \ldots, A_n)$ is called the \emph{mixed determinant} of operators $A_1, \ldots, A_n$.
\end{dfn}

For example, by polarizing the determinant of $2 \times 2$--matrices we obtain
$$
\det \left( \begin{pmatrix}
            a_{11} & a_{12}\\
	    a_{21} & a_{22}
            \end{pmatrix},
\begin{pmatrix}
b_{11} & b_{12}\\
b_{21} & b_{22}
\end{pmatrix}
\right)
= \frac12 (a_{11}b_{22} + a_{22}b_{11} - a_{12}b_{21} - a_{21}b_{12}).
$$

\begin{cor}
$$
\det(\id, B) = \frac12 \tr B
$$
\end{cor}

\begin{lem}
\label{lem:DetSign}
The symmetric bilinear form $\det(\cdot\,, \cdot)$ on the space of $2 \times 2$--matrices has signature $(+,+,-,-)$. The restriction of $\det(\cdot\,, \cdot)$ to the space of symmetric $2 \times 2$--matrices has signature $(+,-,-)$.
\end{lem}
\begin{proof}
The matrices
$$
\begin{pmatrix}
1 & 0\\
0 & 1
\end{pmatrix}, \quad
\begin{pmatrix}
1 & 0\\
0 & -1
\end{pmatrix}, \quad
\begin{pmatrix}
0 & 1\\
1 & 0
\end{pmatrix}, \quad
\begin{pmatrix}
0 & 1\\
-1 & 0
\end{pmatrix}
$$
form an orthogonal basis for $\det(\cdot\,, \cdot)$, the first three of them spanning the space of symmetric matrices. The first and the fourth matrix have a positive determinant, the second and the third a negative.
\end{proof}

\begin{cor}
\label{cor:AB}
Let $A, B$ be symmetric $2 \times 2$--matrices such that
\begin{equation}
\label{eqn:ABOrth}
\det B > 0, \quad \det(A, B) = 0.
\end{equation}
Then $\det A \le 0$. Besides, the equality $\det A = 0$ occurs only if $A = 0$.
\end{cor}
\begin{proof}
Assumptions \eqref{eqn:ABOrth} mean that $B$ is a positive vector for the symmetric bilinear form $\det(\cdot\,, \cdot)$ and that $A$ is orthogonal to $B$. By the second part of Lemma \ref{lem:DetSign}, $\det(\cdot\,, \cdot)$ is negative definite on $B^\perp$. Hence $\det A \le 0$, and the determinant vanishes only if the matrix $A$ vanishes.
\end{proof}

\begin{lem}
\label{lem:DerDet}
Let $B_t \in \End(V)$ be a differentiable family of linear operators. Denote $B = B_0$ and $\dot B = \left.\frac{d}{dt}\right|_{t=0} B_t$. Then
$$
(\det B)^\bigdot = 2 \det(\dot B, B),
$$
where $(\det B)^\bigdot = \left.\frac{d}{dt}\right|_{t=0} (\det B_t).$
\end{lem}
\begin{proof}
Follows easily from bilinearity and symmetry of $\det(\cdot\,, \cdot)$.
\end{proof}

\begin{cor}
\label{cor:DetDotBNeg}
Let $B_t$ be a differentiable family of symmetric $2 \times 2$--matrices, $B_0 = B$. Assume that
$$
\det B > 0, \quad (\det B)^\bigdot = 0.
$$
Then $\det \dot B \le 0$.

Besides, the equality $\det \dot B = 0$ occurs only if $\dot B = 0$.
\end{cor}
\begin{proof}
Use Lemma \ref{lem:DerDet} and Corollary \ref{cor:AB} for $A = \dot B$.
\end{proof}

\subsection{Mixed determinants and tensor algebra}
\label{subsec:MixDetTens}
Recall the definitions of the alternating and symmetrizing operators on the $k$-th tensor power $\bigotimes^k V$ of a vector space $V$:
$$
\Alt(v_1 \otimes \cdots \otimes v_k) := \frac{1}{k!} \sum_{\sigma \in S_k} \sgn \sigma \cdot v_{\sigma(1)} \otimes \cdots \otimes v_{\sigma(k)},
$$
$$
\Sym(v_1 \otimes \cdots \otimes v_k) := \frac{1}{k!} \sum_{\sigma \in S_k} v_{\sigma(1)} \otimes \cdots \otimes v_{\sigma(k)}.
$$

The image of $\Alt$ is denoted by $\bigwedge^k V \subset \bigotimes^k V$ and consists of linear combinations of multivectors
$$
v_1 \wedge \cdots \wedge v_k := k! \Alt(v_1 \otimes \cdots \otimes v_k) = \sum_{\sigma \in S_k} \sgn \sigma \cdot v_{\sigma(1)} \otimes \cdots \otimes v_{\sigma(k)}.
$$

Let $A_1, \ldots, A_k \in \End(V) \cong V \otimes V^*$ be a collection of linear operators. Their tensor product acts naturally as a linear operator on the $k$--th tensor power of $V$:
\begin{equation}
\label{eqn:TensProdOp}
\bigotimes_{i=1}^k A_i \in {\textstyle\bigotimes}^k (V \otimes V^*) \cong \left( {\textstyle\bigotimes}^k V \right) \otimes \left( {\textstyle\bigotimes}^k(V^*) \right) \cong \End \left( {\textstyle\bigotimes}^k V \right),
\end{equation}
because we have $\bigotimes^k (V^*) \cong \left( \bigotimes^k V \right)^*$. Explicitly,
$$
\left( \bigotimes_{i=1}^k A_i \right) (v_1 \otimes \cdots \otimes v_k) := A_1(v_1) \otimes \cdots \otimes A_k(v_k).
$$
If $A_i = A$ for all $i$, then the operator $\bigotimes^k A$ commutes with $\Alt$:
\begin{eqnarray*}
\left( {\textstyle\bigotimes}^k A \right) (\Alt(v_1 \otimes \cdots \otimes v_k)) & = & \frac{1}{k!} \sum_{\sigma \in S_k} \sgn\sigma \cdot A(v_{\sigma(1)}) \otimes \cdots \otimes A(v_{\sigma(k)})\\
& = & \Alt \left( \left( {\textstyle\bigotimes}^k A \right) (v_1 \otimes \cdots \otimes v_k) \right).
\end{eqnarray*}
Consequently,
$$
\left( {\textstyle\bigotimes}^k A \right) (v_1 \wedge \cdots \wedge v_k) = A(v_1) \wedge \cdots \wedge A(v_k),
$$
so that the operator $\bigotimes^k A$ maps $\bigwedge^k V$ to itself. In particular, if $\dim V = n$, then $ \dim {\textstyle\bigwedge}^n V = 1$ and $\bigotimes^n A \big|_{{\scriptstyle\bigwedge}^n V}$ is a multiplication by a scalar. It is well-known that this scalar is the determinant of $A$:
\begin{equation}
\label{eqn:DetATens}{\textstyle\bigotimes}^n A \big|_{{\scriptstyle\bigwedge}^n V} = \det A \cdot \id.
\end{equation}

Now let's come back to the general case \eqref{eqn:TensProdOp}. The operator $\bigotimes_{i=1}^k A_k$ in general does not commute with the alternating operator. Instead, consider the symmetrized tensor product
$$
\Sym\bigotimes_{i=1}^k A_i = \frac{1}{k!} \sum_{\sigma \in S_k} \bigotimes_{i=1}^k A_{\sigma(i)},
$$
that is the image of $\bigotimes_{i=1}^k A_i$ under the symmetrizing operator
$$
\Sym \colon {\textstyle\bigotimes}^k (V \otimes V^*) \to {\textstyle\bigotimes}^k (V \otimes V^*).
$$
A simple computation shows that $\Sym\bigotimes_{i=1}^k A_i$ commutes with the alternating operator on $\bigotimes^k V$ and acts on $\bigwedge^k V$ as
\begin{eqnarray*}
\left( \Sym\bigotimes_{i=1}^k A_i \right) (v_1 \wedge \cdots \wedge v_k) & = & \frac{1}{k!} \sum_{\sigma, \tau \in S_k} \left( \sgn\tau \cdot \bigotimes_{i=1}^k A_{\sigma(i)}(v_{\tau(i)}) \right)\\
& = & \frac{1}{k!} \sum_{\sigma \in S_k} A_{\sigma(1)}(v_1) \wedge \cdots \wedge A_{\sigma(k)}(v_k).
\end{eqnarray*}
For $k = n = \dim V$ we have the following lemma that can serve as an alternative definition of the mixed determinant.

\begin{lem}
\label{lem:MixDetTens}
If $\dim V = n$ and $A_1, \ldots, A_n \in \End(V)$, then
\begin{equation}
\label{eqn:SymAMixDet}
\left( \Sym\bigotimes_{i=1}^n A_i \right) \Bigg|_{{\scriptstyle\bigwedge}^n V} = \det(A_1, \ldots, A_n) \cdot \id.
\end{equation}
\end{lem}
\begin{proof}
If $A_i = A$ for all $i$, then we have $\Sym \bigotimes^n A = \bigotimes^n A$, and \eqref{eqn:SymAMixDet} turns into \eqref{eqn:DetATens}. Both sides in \eqref{eqn:SymAMixDet} are symmetric and polylinear in $(A_i)$: the left hand side by construction, the right hand side by definition. Therefore the equality holds for all $(A_i)$.
\end{proof}

\begin{rem}
\label{rem:SymAlt}
Instead of \emph{applying} symmetrization \emph{to} $\bigotimes_{i=1}^k A_i$, one can \emph{compose} alternation \emph{with} $\bigotimes_{i=1}^k A_i$. This gives the same operator when restricted to $\bigwedge^k V$:
\begin{equation}
\label{eqn:SymAlt}
\left. \left( \Alt \circ \bigotimes_{i=1}^k A_i \right) \right|_{{\scriptstyle\bigwedge}^k V} = \left. \left( \Sym\bigotimes_{i=1}^k A_i \right) \right|_{{\scriptstyle\bigwedge}^k V}.
\end{equation}
\end{rem}

\subsection{Vector-valued forms}
\label{subsec:MixDetVVal}
Let $V$ and $W$ be vector spaces.

\begin{dfn}
\label{dfn:WedgProd}
A $V$--valued $k$--form on $W$ is an element of $V \otimes \bigwedge^k W^*$. For two $V$--valued forms $\omega \in V \otimes {\textstyle\bigwedge}^k W^*$ and $\eta \in V \otimes {\textstyle\bigwedge}^l W^*$, their wedge product
$$
\omega \wedge \eta \in (V \otimes V) \otimes {\textstyle\bigwedge}^{k+l} W^*
$$
is defined as the image of $\omega \otimes \eta$ under the map
$$
\id \otimes \, \wedge \colon (V \otimes V) \otimes \left( {\textstyle\bigwedge}^k W^* \otimes {\textstyle\bigwedge}^l W^* \right) \to (V \otimes V) \otimes {\textstyle\bigwedge}^{k+l} W^*
$$
that takes the wedge product of $\bigwedge W^*$--components.
\end{dfn}

The wedge product of several $V$--valued forms is defined in a similar way.

Let $n = \dim V$, and let $\dvol \in \bigwedge^n (V^*)$ be a distinguished element, a volume form. 

\begin{dfn}
\label{dfn:DVolForm}
Let $\omega \in \left(\bigotimes^n V \right) \otimes \bigwedge^k W^*$ be a $\left( \bigotimes^n V \right)$--valued $k$--form on~$W$. Define a $k$--form
$$
\dvol(\omega) \in {\textstyle\bigwedge}^k W^*
$$
as the image of $\omega$ under the linear map
$$
\dvol \otimes \id \colon \Big( {\textstyle\bigotimes}^n V \Big) \otimes \left( {\textstyle\bigwedge}^k W^* \right) \to {\textstyle\bigwedge}^k W^*.
$$
\end{dfn}

In particular, for $\omega_i \in V \otimes \bigwedge^{k_i} W$, $i = 1, \ldots, n$ we have a $k$--form
$$
\dvol(\omega_1 \wedge \cdots \wedge \omega_n) \in {\textstyle\bigwedge}^k W,
$$
where $k = \sum_{i=1}^n k_i$. In a special case $V = W, k_i =1$ this construction is related to the mixed determinant of linear operators. Indeed, each
$$
\omega_i \in V \otimes {\textstyle\bigwedge}^1 V \cong V \otimes V^*.
$$
is a linear operator on $V$, and we have the following lemma.

\begin{lem}
\label{lem:DvolDet}
Let $\omega_i \in V \otimes V^*$, $i = 1, \ldots, n$, be $V$--valued 1--forms on an $n$--dimensional vector space $V$. Then
\begin{equation}
\label{eqn:DvolDet}
\dvol(\omega_1 \wedge \cdots \wedge \omega_n) = n! \det(\omega_1, \ldots, \omega_n) \cdot \dvol,
\end{equation}
where on the right hand side stands the mixed determinant of linear operators $\omega_i \in \End V$.
\end{lem}
\begin{proof}
The proof goes by inverting the roles of $V$ and $V^*$ in Lemma \ref{lem:MixDetTens} and Remark \ref{rem:SymAlt}.

Since $V \otimes V^* \cong \End(V^*)$, we can view $\omega_i$ as a linear operator on $V^*$. It is easy to see that
$$
\omega_1 \wedge \cdots \wedge \omega_n = n! \Alt \circ \bigotimes_{i=1}^n \omega_i \in \End \left( {\textstyle\bigotimes}^n V^* \right).
$$
Due to \eqref{eqn:SymAlt} and \eqref{eqn:SymAMixDet}, the restriction of $\omega_1 \wedge \cdots \wedge \omega_n$ to $\bigwedge^n V^*$ is multiplication with $n! \det(\omega_1, \ldots, \omega_n)$. On the other hand, $\dvol(\omega_1 \wedge \cdots \wedge \omega_n)$ is by definition the value of $\omega_1 \wedge \cdots \wedge \omega_n$ on $\dvol$. Equation \eqref{eqn:DvolDet} follows.
\end{proof}

More generally, if $\omega_i \in V \otimes W^* \cong \Hom(W,V)$ with $\dim W = \dim V = n$, then we have
$$
\dvol_V(\omega_1 \wedge \cdots \wedge \omega_n) = n! \det(\omega_1, \ldots, \omega_n) \cdot \dvol_W,
$$
where $\dvol_V$ and $\dvol_W$ are some volume forms on $V$, respectively $W$, and the mixed determinant is the polarization of $\det \colon \Hom(W,V) \to \R$ which is defined with respect to $\dvol_V$ and $\dvol_W$. In the special case $\omega_i = \omega$ we obtain
\begin{equation}
\label{eqn:DVolVW}
\dvol_V \left( {\textstyle\bigwedge}^n \omega \right) = n! \det \omega \cdot \dvol_W.
\end{equation}
One sees immediately that
$$
\dvol_V \left( {\textstyle\bigwedge}^n \omega \right) = n! \omega^*(\dvol_V).
$$
Therefore equation \eqref{eqn:DVolVW} is just another form of a well-known identity
$$
\omega^*(\dvol_V) = \det \omega \cdot \dvol_W.
$$

\begin{lem}
\label{lem:SymmDet}
Let $\omega_i \in V \otimes \bigwedge^{k_i} V^*$ be a $V$--valued form on $V$, $i = 1, \ldots, n$. Then we have
$$
\dvol(\omega_1 \wedge \omega_2 \wedge \cdots \wedge \omega_n) = (-1)^{k_1 k_2 + 1} \dvol(\omega_2 \wedge \omega_1 \wedge \cdots \wedge \omega_n),
$$
and similarly for any other transposition of two factors.
\end{lem}
\begin{proof}
As with usual forms with values in a field, transposing two factors in $\omega_1 \wedge \cdots \wedge \omega_n$ exchanges two blocks of factors in $\bigwedge^k V^*$ which causes the factor $(-1)^{k_i k_j}$ to appear. In our case there is also a transposition of two factors in $\bigotimes^n V$ which changes the sign of the evaluation of $\dvol$.
\end{proof}

\subsection{Vector bundle-valued differential forms}
\label{subsec:BundleVal}
Let $M$ be a smooth manifold, and let $E \to M$ be a smooth finite-dimensional vector bundle over $M$.

\begin{dfn}
An \emph{$E$--valued differential $k$--form} on $M$ is a section of the vector bundle $E \otimes \bigwedge^k (T^*M)$. The set of all $E$--valued differential $k$--forms is denoted by $\Omega^k(M, E)$.
\end{dfn}

The wedge product of two vector bundle-valued differential forms
$$
\wedge \colon \Omega^k(M, E) \otimes \Omega^l(M, F) \to \Omega^{k+l}(M, E \otimes F),
$$
is defined as in Definition \ref{dfn:WedgProd}. If $\dim E = n$, and $E$ is equipped with a volume form $\dvol \in \Gamma(\bigwedge^n E)$, then we define
$$
\dvol \colon \Omega^k \left( M, {\textstyle\bigotimes}^n E \right) \to \Omega^k (M)
$$
as in Definition \ref{dfn:DVolForm}.

A connection $\nabla$ on $E$ induces exterior differentiation of $E$--valued differential forms:
$$
d^\nabla \colon \Omega^k(M,E) \to \Omega^{k+1}(M,E),
$$
see e.~g. \cite[Chapter 3.1]{Jost08}. A connection $\nabla$ induces also connections on all tensor bundles associated with $E$, and thus exterior differentiation of corresponding tensor bundle-valued differential forms. All these definitions go by postulating certain Leibniz rules and imply the following two lemmas.

\begin{lem}
\label{lem:LeibVect}
Let $\omega_i \in \Omega^{k_i}(M,E)$, $i = 1, \ldots, p$ be $E$--valued differential forms on $M$. Then we have
$$
d^\nabla (\omega_1 \wedge \cdots \wedge \omega_p) = \sum_{i=1}^p (-1)^{k_1+\cdots+k_{i-1}}  (\omega_1 \wedge \cdots \wedge d^\nabla\omega_i \wedge \cdots \wedge \omega_p),
$$
where $d^\nabla$ are the exterior derivatives on $\Omega(M,E)$ and $\Omega(M, \bigotimes^p E)$.
\end{lem}

\begin{lem}
\label{lem:DDet}
Assume that the volume form $\dvol$ is parallel:
$$
\nabla (\dvol) = 0.
$$
Then we have
$$
d (\dvol(\omega)) = \dvol \left( d^\nabla\omega \right)
$$
for every $\omega \in \Omega^k(M, {\textstyle\bigotimes}^n E)$.
\end{lem}

In particular, the assumption of Lemma \ref{lem:DDet} holds if the vector bundle $E$ is equipped with a scalar product, $\dvol$ is the associated volume form, and $\nabla$ is a metric connection. A natural situation when this occurs is a smooth submanifold $M$ of a Riemannian manifold $N$ with $E = TN|_M$, the restriction of the tangent bundle of $N$ to $M$. More specifically, $N$ may be $M$ itself, with $\nabla$ the Levy-Civita connection on $E = TM$.

\subsection{Vector-valued differential forms}
\label{subsec:VecVal}
Assume that $E$ is trivial and a trivialization
\begin{equation}
\label{eqn:Triv}
E \cong V \times M,
\end{equation}
is fixed. This induces an isomorphism
$$
\Omega^k(M, E) \cong C^\infty(M, V) \otimes \Omega^k M
$$
which allows us to call differential forms with values in a \emph{trivialized} bundle \emph{vector-valued differential forms}. A choice of a basis in $V$ establishes an isomorphism
$$
C^\infty(M, V) \otimes \Omega^k M \cong (\Omega^k M)^n,
$$
where $n = \dim V$.

Let $\nabla$ be the connection associated with the trivialization \eqref{eqn:Triv}. The associated exterior derivative $d^\nabla$ is nothing else than the componentwise differentiation
$$
(\Omega^k M)^n \to (\Omega^{k+1} M)^n
$$
with respect to an arbitrary choice of a basis in $V$.

A trivialization \eqref{eqn:Triv} induces also a trivialization of the dual bundle $E^*$. Choose an element $\dvol \in \bigwedge^n V^*$ and consider the constant section $\dvol$ of the bundle $\bigwedge^n E^*$. Then we have $\nabla (\dvol) = 0$, which implies that Lemma \ref{lem:DDet} holds for the canonical connection on a trivialized bundle.

The next lemma states that a trivial connection if \emph{flat}, which is false for general connections.

\begin{lem}
\label{lem:DD}
Let $d^\nabla$ be the exterior derivative associated with the canonical flat connection on a trivialized vector bundle \eqref{eqn:Triv}. Then we have
$$
d^\nabla (d^\nabla \omega) = 0,
$$
for every $V$--valued differential form $\omega$.
\end{lem}
\begin{proof}
Immediate from the interpretation of $d^\nabla$ as the componentwise exterior differentiation.
\end{proof}

A natural situation when we have to do with a trivialized vector bundle is an embedded manifold $M \subset \R^n$ with the vector bundle $E = T\R^n|_M$.

\section{Curvature of $\rho^2$-warped product metrics}
\label{sec:CurvWarped}
Let $\hat g$ be a Riemannian metric on $\Sph^2$. Consider the Riemannian metric
\begin{equation}
\label{eqn:WarpMetr}
\widetilde{g} = d\rho^2 + \rho^2 \hat g
\end{equation}
on the manifold $\R_+ \times \Sph^2$. In particular, the metric $\widetilde g_r$ in \eqref{eqn:TildeGR} is of this form.

We compute here the Riemann tensor and the sectional curvatures of the metric $\wg$. In more general settings, this is done in \cite[Section 7.42]{ON83} and \cite[Section 6, Exercise 11]{Kue99}.

Let $\widetilde R$ be the curvature tensor of the Riemannian manifold $(\R_+ \times \Sph^2, \widetilde g)$. Let $\partial_\rho$ be the unit vector field arising from the product structure on $\R_+ \times \Sph^2$. By $\partial_\rho^\perp \subset T_{(\rho, x)}(\R_+ \times \Sph^2)$ we denote the plane orthogonal to $\partial_\rho$. As $\partial_\rho$ is also the gradient of the function $\rho$, the plane $\partial_\rho^\perp$ is tangent to the level sets of $\rho$. Denote by $\sec_{(\rho, x)}(L)$ the sectional curvature of $\wg$ in the plane $L \subset T_{(\rho, x)}(\R_+ \times \Sph^2)$.

\begin{lem}
\label{lem:WarpProdCurv}
The curvature of the Riemannian metric \eqref{eqn:WarpMetr} has the following properties.
\begin{equation}
\label{eqn:R}
\wR(X,Y)Z = \sec(\partial_\rho^\perp) \cdot \dvol(\partial_\rho, X, Y) \cdot (Z \times \partial_\rho)
\end{equation}
for all vectors $X, Y, Z \in T_{(\rho,x)}(\R_+ \times \Sph^2)$;
\begin{equation}
\label{eqn:SecL}
\sec_{(\rho,x)}(L) = \cos^2\phi \cdot \sec_{(\rho,x)}(\partial_\rho^\perp),
\end{equation}
where $\phi$ is the angle between the planes $L$ and $\partial_\rho^\perp$;
\begin{equation}
\label{eqn:Sec1}
\sec_{(\rho,x)}(\partial_\rho^\perp) = \frac{\sec_{(1,x)}(\partial_\rho^\perp)}{\rho^2}.
\end{equation}
\end{lem}
\begin{proof}
The curvature tensor can be computed following \cite[Chapter 2.4]{Pet06}. For this, consider the function $\rho$ on $\R_+ \times \Sph^2 \to \R$, given by the projection to the first factor. This is a distance function, i.e. its gradient has norm 1. The Hessian
$$
\widetilde\Hess \rho: X \mapsto \wnabla_X \wnabla \rho
$$
can easily be computed:
$$
\widetilde \Hess \rho = \frac{1}{\rho} \pi_\rho,
$$
where $\pi_\rho$ is the orthogonal projection to $\partial_\rho^\perp$. Then the radial curvature equation yields
$$
\wR(\,\cdot\,, \partial_\rho) \partial_\rho = 0,
$$
and the mixed curvature equation (Codazzi-Mainardi equation) implies
$$
\widetilde{g}(\wR(X,Y)Z, \partial_\rho) = 0,
$$
for any vectors $X, Y, Z \in \partial_\rho^\perp$.
It follows that
\begin{equation}
\label{eqn:RXX'}
\widetilde{g}(\wR(X,Y)Z, W) = \widetilde{g}(\wR(X',Y')Z', W'),
\end{equation}
for arbitrary vectors $X, Y, Z, W$, where $X' = \pi_\rho(X)$ and so on. As $X'$, $Y'$, $Z'$, and $W'$ all lie in the plane $\partial_\rho^\perp$, we have
\begin{equation}
\label{eqn:RSec}
\widetilde{g}(\wR(X', Y')Z', W') = \sec(\partial_\rho^\perp) \cdot \darea(X', Y') \cdot \darea(W', Z'),
\end{equation}
where $\darea$ denotes the area form in $\partial_\rho^\perp$ induced from $\widetilde{g}$. Clearly,
\begin{eqnarray*}
\darea(X',Y') & = & \dvol(\partial_\rho, X, Y),\\
\darea(W',Z') & = & \dvol(\partial_\rho, W, Z) = \widetilde{g}(Z \times \partial_\rho, W).
\end{eqnarray*}
By substituting this in \eqref{eqn:RSec} and using \eqref{eqn:RXX'}, we obtain \eqref{eqn:R}.

\end{proof}


\end{document}

%% file: Gradients.tex
\begin{picture}(0,0)%
\includegraphics{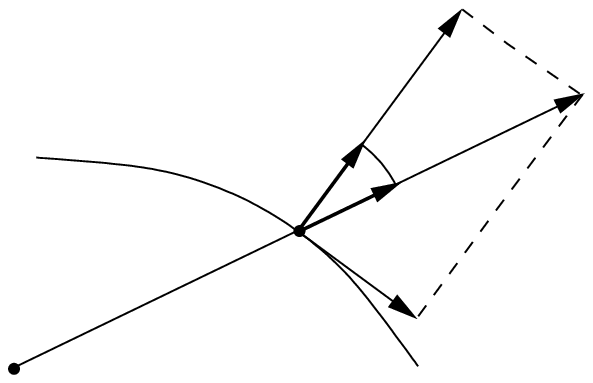}%
\end{picture}%
\setlength{\unitlength}{4144sp}%
\begingroup\makeatletter\ifx\SetFigFont\undefined%
\gdef\SetFigFont#1#2#3#4#5{%
  \reset@font\fontsize{#1}{#2pt}%
  \fontfamily{#3}\fontseries{#4}\fontshape{#5}%
  \selectfont}%
\fi\endgroup%
\begin{picture}(2756,1901)(1948,-1612)
\put(2150,-1557){\makebox(0,0)[lb]{\smash{{\SetFigFont{10}{12.0}{\rmdefault}{\mddefault}{\updefault}{\color[rgb]{0,0,0}$0$}%
}}}}
\put(3816,-755){\makebox(0,0)[lb]{\smash{{\SetFigFont{10}{12.0}{\rmdefault}{\mddefault}{\updefault}{\color[rgb]{0,0,0}$\partial_\rho$}%
}}}}
\put(4689,-332){\makebox(0,0)[lb]{\smash{{\SetFigFont{10}{12.0}{\rmdefault}{\mddefault}{\updefault}{\color[rgb]{0,0,0}$\widetilde\nabla f$}%
}}}}
\put(3931,-1320){\makebox(0,0)[lb]{\smash{{\SetFigFont{10}{12.0}{\rmdefault}{\mddefault}{\updefault}{\color[rgb]{0,0,0}$\nabla f$}%
}}}}
\put(3482,-451){\makebox(0,0)[lb]{\smash{{\SetFigFont{10}{12.0}{\rmdefault}{\mddefault}{\updefault}{\color[rgb]{0,0,0}$\nu$}%
}}}}
\put(3884,166){\makebox(0,0)[lb]{\smash{{\SetFigFont{10}{12.0}{\rmdefault}{\mddefault}{\updefault}{\color[rgb]{0,0,0}$h\nu$}%
}}}}
\put(3775,-477){\makebox(0,0)[lb]{\smash{{\SetFigFont{10}{12.0}{\rmdefault}{\mddefault}{\updefault}{\color[rgb]{0,0,0}$\alpha$}%
}}}}
\put(1963,-554){\makebox(0,0)[lb]{\smash{{\SetFigFont{10}{12.0}{\rmdefault}{\mddefault}{\updefault}{\color[rgb]{0,0,0}$M_r$}%
}}}}
\end{picture}%